\documentclass{amsart}

\usepackage{amssymb}
\usepackage{amsmath}
\usepackage{amsthm}
\usepackage{amsbsy}
\usepackage{bm}
\usepackage{hyperref}
\date{\today}
\usepackage{cite}
\usepackage{array}
\date{\today}
\newcommand{\comment}[1]{}

\theoremstyle{theorem}
    \newtheorem{theorem}{Theorem}
    \newtheorem{lemma}[theorem]{Lemma}

\theoremstyle{definition} 
    
    \newtheorem{fact}[theorem]{Fact}
    \newtheorem{result}[theorem]{Result}
    \newtheorem{remark}[theorem]{Remark}
    \newtheorem{example}[theorem]{Example}
    \newtheorem{exercise}[theorem]{Exercise}

\def\C{\mathbb{C}}
\def\D{\mathbb{D}}

\def\X{{\mathcal X}}

\def\<{\langle}
\def\>{\rangle}

\newcommand{\E}{\mbox{\bf E}}

\def\bar{\overline}
\def\P{{\bf P}}

\def\d{\partial}

\newcommand\mnote[1]{} 
\newcommand\be{\begin{equation*}}

\newcommand\ee{\end{equation*}}

\newcommand\ben{\begin{equation}}
\newcommand\een{\end{equation}}
\newcommand\bes{\begin{eqnarray*}}
\newcommand\ees{\end{eqnarray*}}

\newcommand\bex{\begin{exercise}}
\newcommand\eex{\end{exercise}}
\newcommand\beg{\begin{example}}
\newcommand\eeg{\end{example}}
\newcommand\benu{\begin{enumerate}}
\newcommand\eenu{\end{enumerate}}
\newcommand\beit{\begin{itemize}}
\newcommand\eeit{\end{itemize}}
\newcommand\berk{\begin{remark}}
\newcommand\eerk{\end{remark}}
\newcommand\bdefn{\begin{defintion}}
\newcommand\edefn{\end{definition}}
\newcommand\bthm{\begin{theorem}}
\newcommand\ethm{\end{theorem}}
\newcommand\bprf{\begin{proof}}
\newcommand\eprf{\end{proof}}
\newcommand\blem{\begin{lemma}}
\newcommand\elem{\end{lemma}}

\newcommand{\sm}{{\raise0.3ex\hbox{$\scriptstyle \setminus$}}}

\def\CHI{\mathchoice%
{\raise2pt\hbox{$\chi$}}%
{\raise2pt\hbox{$\chi$}}%
{\raise1.3pt\hbox{$\scriptstyle\chi$}}%
{\raise0.8pt\hbox{$\scriptscriptstyle\chi$}}}
\def\smalloplus{\raise1pt\hbox{$\,\scriptstyle \oplus\;$}}

\title{Hole probabilities for finite and infinite Ginibre ensembles }

\author{Kartick Adhikari}

\address{Department of Mathematics\\
        Indian Institute of Science\\
        Bangalore 560012, India}

\email{kartickmath@math.iisc.ernet.in}

\author{Nanda Kishore Reddy}

\address{Department of Mathematics\\
        Indian Institute of Science\\
        Bangalore 560012, India}

\email{kishore11@math.iisc.ernet.in}

\thanks{Partially supported by UGC Centre for Advanced Studies. Research of Nanda Kishore reddy is supported by CSIR-SPM fellowship, CSIR, Government of India.}

\begin{document}
\maketitle

\begin{abstract}
We study the hole probabilities of the infinite Ginibre ensemble ${\X}_{\infty}$, a determinantal point process on the complex plane   with the kernel $\mathbb K(z,w)= \frac{1}{\pi}e^{z\bar w-\frac{1}{2}|z|^2-\frac{1}{2}|w|^2}$ with respect to the Lebesgue measure on the complex plane. Let $U$ be an open subset of open unit disk $\D$ and ${\X}_{\infty}(rU)$ denote the number of points  of ${\X}_{\infty}$ that fall in $rU$. Then, under some  conditions on $U$, we  show that 
$$
\lim_{r\to \infty}\frac{1}{r^4}\log\P[\X_{\infty}(rU)=0]=R_{\emptyset}-R_{U},
$$
where $\emptyset$ is the empty set and 
$$
R_U:=\inf_{\mu\in \mathcal P(U^c)}\left\{\iint \log{\frac{1}{|z-w|}}d\mu(z)d\mu(w)+\int
|z|^2d\mu(z) \right\},
$$ 
$\mathcal P(U^c)$ is the space of all compactly supported  probability measures with support in $U^c$. Using potential theory, we give an explicit formula for $R_U$, the minimum possible energy of a probability measure compactly supported on $U^c$ under logarithmic potential with a quadratic external field. Moreover, we calculate $R_U$ explicitly for some special sets like annulus, cardioid, ellipse, equilateral triangle and half disk.  
\end{abstract}
\section{Introduction and main results}
Let  $\X$ be a point process (see \cite{verejones}, p. 7) on $\C$ and let $U$ be an open set in $\C$. The probability that $U$ contains no points of $\X$ is called {\it hole}/{\it gap  probability}  for $U$. Our main aim, in this paper,  is to compute   hole probabilities of finite and infinite Ginibre ensembles  for various open sets $U$, using potential theory. For earlier studies on hole probabilities for various other point processes, we refer the reader to  \cite{akemannhole}, \cite{nishry10}, \cite{nishry11}, \cite{nishry12}, \cite{sodin}. 

A $n\times n$ random matrix $G_{n}$, with i.i.d. standard  complex Gaussian entries, is called $n$-th complex \textit{Ginibre ensemble}.  The joint probability density function of the eigenvalues  of $G_n$ (see \cite{ginibre}, \cite{manjubook}, p. 60) is
\begin{equation*}\label{eqn:density}
\frac{1}{\pi^n \prod_{k=1}^{n} k!}e^{-\sum_{k=1}^{n}|z_k|^2}\prod_{i<j} |z_i-z_j|^2,
\end{equation*}
where $z_1,z_2,\ldots,z_n\in \C$. The point process $\X_n$, constituting the eigenvalues of $G_{n}$, is a   determinantal point process (see \cite{manjubook}, p. 48, \cite{andersonbook}, p. 215), as the joint density of the eigenvalues  can also be written as
 \begin{equation*}\label{eqn:density1}
\frac{1}{n!} \det \left(\mathbb K_{n}(z_i,z_j) \right)_{1 \le i,j \le n},
\end{equation*}
 with the kernel 
 $$\mathbb K_{n}(z,w)=\sum_{k=1}^{n} \varphi_k(z)\overline{\varphi_k(w)}, \mbox{  where 
 $\varphi_k(z)= \frac{1}{\sqrt{{\pi}(k-1)!}} z^{k-1} e^{-\frac{1}{2}|z|^2} $},$$
  with respect to the background measure $dm(z)$, where $m$ denotes the Lebesgue measure on the complex plane. 
  
  The {\it infinite Ginibre ensemble} $\X_{\infty}$ (see \cite{manjubook}, p. 60) on the complex plane is a determinantal point process  with the kernel $$\mathbb K(z,w)= \frac{1}{\pi}e^{z\overline{w} -\frac{1}{2}|z|^2-\frac{1}{2}|w|^2}$$ with respect to the background measure $dm(z)$. Since $\mathbb K_{n}(z,w)$ converges to $\mathbb K(z,w)$ as $n \to \infty$,   the  $n$-th complex Ginibre ensemble converges in distribution to infinite Ginibre ensemble as $n \to \infty$ (see \cite{verejones}, Theorem 11.1.VII).

\noindent{\bf Notation: }Define $rU:= \{rz|z \in U\}$. $\partial U$ denotes the boundary of $U$. The number of points of a point process $\X$ that fall in $U$ is denoted by $\X(U)$.  For  $E \subset \mathbb C$, $\mathcal P (E)$ denotes
the space of all compactly supported probability measures with support in $ E$. 

The main purpose of this paper is to compute  the   limits of hole probabilities for infinite Ginibre ensemble for a wide class of open sets $U$. It is known, about the hole probabilities for infinite Ginibre ensemble $\X_{\infty}$, that
$$
\lim_{r\to \infty}\frac{1}{r^4}\log\P[\X_{\infty}(r\D)=0]=-\frac{1}{4},
$$
where $\D$ is open unit disk. This has been computed in \cite{shirai} (see Theorem 1.1). An alternate proof of this has been obtained in \cite{manjubook}, Proposition 7.2.1. The key idea  in  the proof (\cite{manjubook}, Proposition 7.2.1) is  that the set of absolute values of the points of $\X_{\infty}$  has the same distribution as $\{R_1,R_2,\ldots \},$ where $R_k^2\sim \mbox{Gamma}(k,1)$ and all the $R_k$s are independent. This fact is due to Kostlan \cite{kostlan} (see Result \ref{lem:kostlan} below). By using the same idea we have the following result for annulus.

\begin{theorem}\label{thm:annulus1}
Let $U_c=\{z \;|\; c<|z|<1\}$ for fixed $0<c<1$. Then
$$
\lim_{r\to \infty}\frac{1}{r^4}\log \P[\X_{\infty}(rU_c)=0]=-\frac{(1-c^2)}{4}\cdot\left(1+c^2+\frac{1-c^2}{\log c}\right).
$$
\end{theorem}

 Observe that  the above  idea cannot be  applied for non circular domains. So, to prove the main theorem of this paper (Theorem \ref{thm:infinite}), we resort to the computation of hole probabilities of finite Ginibre ensembles, as the infinite Ginibre ensemble is the distributional limit of finite Ginibre ensembles. Hole probability of $n$-th Ginibre ensemble for an open set $\sqrt{n} U$ is given by 
 \begin{align}\label{eqn:holeprob}
 \P[\X_n(\sqrt{n} U)=0]=\frac{1}{Z_n}\int_{ U^c}\ldots \int_{ U^c}  e^{-n\sum_{k=1}^{n}|z_k|^2}\prod_{i<j}|z_i-z_j|^2\prod_{i=1}^n dm(z_i),
 \end{align}
where $Z_n=n^{-\frac{n^2}{2}}{\pi^n \prod_{k=1}^{n} k!}$.  Circular law \cite{ginibre} tells us that the empirical eigenvalue distribution $\rho_n$  of $\frac{1}{\sqrt{n}}G_n$ converges to the uniform measure on unit disk $\D$ as $n \to \infty$. So, for  $U \subset \D$, $\P[\X_n(\sqrt{n} U)=0]$ converges to zero as $n \to \infty$. Observe  that $\P[\X_n(\sqrt{n} U)=0]=\P[\rho_n \in \mathcal P (U^c) ]$. Therefore by  Large deviation principle for the empirical eigenvalue distribution of Ginibre ensemble, proved in \cite{hiai}, we have an upper bound for the limits of hole probabilities,
 $$
\limsup_{n\to
\infty}\frac{1}{n^2}\log \P[\X_n(\sqrt{n} U)=0]\le-\inf_{\mu\in
\mathcal P (U^c)}R_{\mu}+\frac{3}{4},
$$
 where the rate function $R_{\mu}$ is  the following functional on $\mathcal P (\C)$
\begin{equation*}
R_{\mu}=\iint \log\frac{1}{|z-w|}d\mu(z)d\mu(w)+\int
|z|^2d\mu(z),
\end{equation*}
as the set $\mathcal P (U^c)$ is closed in $\mathcal P (\C)$ with weak topology. No non-trivial lower bound for hole probabilities  can be deduced from the large deviation principle, as the set $\mathcal P (U^c)$ has empty interior. See that, for $a \in U$ and $\mu \in\mathcal P (U^c)$, $(1-\frac{1}{n})\mu + \frac{1}{n} \delta_{a} \notin \mathcal P (U^c)$ for all $n$ and converges to $\mu$ as $n \to \infty$. Nonetheless we have the following lemma which gives a good  lower  bound for the hole probabilities of finite Ginibre ensembles.

\begin{lemma}\label{thm:holeprobability}
Let $U$ be a open subset of $\mathbb C$ and $\X_n$ be the point process of 
eigenvalues of ${G_n}$. Then
$$
\limsup_{n\to
\infty}\frac{1}{n^2}\log\P[\X_n(\sqrt{n} U)=0]\le-\inf_{\mu\in
\mathcal P (U^c)}R_{\mu}+\frac{3}{4},
$$
$$
\liminf_{n\to
\infty}\frac{1}{n^2}\log \P[\X_n(\sqrt{n} U)=0]\ge-\inf_{\mu\in
\mathcal A}R_{\mu}+\frac{3}{4},
$$
where 
$\mathcal A=\{\mu\in \mathcal P(\mathbb C): \mbox {dist(Supp($\mu$),$\overline{U}
$)$>0$}\}$.
\end{lemma}

Notice that $\mathcal P (U_{\epsilon}^c) \subset \mathcal A$ for every $\epsilon > 0$, where $U_{\epsilon}$ is $\epsilon$-neighbourhood  of $U$. This lemma requires us to study 
$$
R_U:=\inf_{\mu\in \mathcal P (U^c)}R_{\mu}
$$
for all open sets $U \subset \D$, to see if the upper and lower bounds in the above lemma match.  The measure for which this infimum is attained is called {\it  equilibrium measure} under logarithmic potential with quadratic  external field.  If $U=\emptyset$, then it is known \cite{allez}, \cite{armstrong} that the equilibrium measure is the uniform probability measure on the unit disk and $R_{\emptyset}$ is $\frac{3}{4}$.
For the sake of completeness, we provide a proof of this in Section \ref{proof}.   For a class of open sets $U$ with certain boundary conditions, equilibrium measures have already  been described in \cite{armstrong}. 
Our next result, using a formulation different from that of \cite{armstrong}, provides a  formula to compute the constant $R_U$ for any open set ${U} \subseteq {\mathbb D}$  and describes the equilibrium measure in terms of the balayage measure on $\partial U$, the definition of which is as follows. 

\noindent{\bf Balayage measure: } For a bounded open set $U$, there exists a  unique measure $\mu$ on $\partial U$ such that $\mu(B)=0$ for every   Borel polar set $B \subset \C$ and 
\begin{eqnarray*}
\int_{\partial U}\log \frac{1}{|z-w|}d\mu(w)=\frac{1}{\pi}\int_{{U}}\log \frac{1}{|z-w|}dm(w)\mbox{ for quasi-every } z\in {{\mathbb C}}\backslash {U}.
\end{eqnarray*} 
$\mu$ is said to be the {\it balayage} measure associated with  measure  $\frac{1}{\pi}m$ on $ U$. In this paper we simply call it the {\it balayage } measure on $\partial {U}$.  For discussion on the balayage measures associated with general measures on $\mathbb{C}$, we refer the  reader to  \cite{totikbook} (p. 110).
 
\begin{theorem}\label{thm:generalformula}
Let $U$ be an open set such that ${U} \subseteq {\mathbb D}$.  Then the equilibrium measure for $U^c$, under logarithmic potential with quadratic  external field, is
$\nu=\nu_1+\nu_2$ and
$$
R_U=\frac{3}{4}+\frac{1}{2}\left[\int_{\partial U}|z|^2d\nu_2(z)-\frac{1}{\pi}\int_{{U}}|z|^2dm(z)\right],
$$
where 
\begin{eqnarray*}
d\nu_1(z)&=&\left\{\begin{array}{lr}
\frac{1}{\pi}dm(z)& \mbox{  if  } z\in\overline{\mathbb D}\backslash {U}\\0& \mbox{o.w.}
\end{array}\right.
\end{eqnarray*}
and  $\nu_2$ is the balayage measure on $\partial U$. 
\end{theorem}
 The above lemma and theorem give us limits of hole probabilities of finite Ginibre ensembles for a certain class of sets $U$. 

\begin{theorem}\label{thm:holeprobability12}
Let $\X_n$ be the point process of 
eigenvalues of ${G_n}$ and let $U\subseteq \D$ be an open set  such that
there exists a sequence of open sets ${U_n}$ such that $\overline
U \subset U_n  \subseteq \D$ for all $n$ and the balayage measure $\nu_n$
on $\partial U_n$ converges weakly to the balayage measure $\nu$ on
$\partial U$.
Then
$$
\lim_{n\to
\infty}\frac{1}{n^2}\log\P[\X_n(\sqrt{n}U)=0]=R_{\emptyset}-R_U.
$$
\end{theorem}
 
The second class of sets $U$ we consider satisfy the exterior ball condition, i.e.,
there exists $\epsilon > 0$  such that for every $z\in \partial U$ there exists a 
$\eta\in U^c$ such that
\begin{eqnarray}\label{eqn:condition}
U^c\supset B(\eta,\epsilon)\;\; \mbox{and
}\; |z-\eta |= \epsilon.
\end{eqnarray}
 Note that all convex domains satisfy the
condition  \eqref{eqn:condition}. Annulus is not a convex domain but
it satisfies the condition \eqref{eqn:condition}. The following theorem gives hole probabilities for such open sets.

\begin{theorem}\label{thm:holeprobability1}
Let $\X_n$ be the point process of eigenvalues of
$G_n$ and let $U\subseteq \D$ be an open set satisfying   condition
\eqref{eqn:condition}. Then
$$
\lim_{n\to \infty}\frac{1}{n^2}\log\P[\X_n(\sqrt{n}U)=0]=R_{\emptyset}-R_U ,
$$
where  $\X_n(U)$ denotes the number of points of $\X_n$ that fall in $U$.
\end{theorem}
 
 The above theorem doesn't include cases of cardioid or sectors with obtuse angle at centre. Theorem \ref{thm:holeprobability12} takes care of these and all the other sets $U$ which can contain scaled-down copies of themselves. But Theorem \ref{thm:holeprobability12}, unlike Theorem \ref{thm:holeprobability1}, requires the boundary of $U$ to not intersect the unit circle.  The proof of Theorem \ref{thm:holeprobability1} makes use of Fekete points, whereas that of Theorem  \ref{thm:holeprobability12} makes use of the balayage measure. We generalize   Theorem \ref{thm:holeprobability12} and  Theorem  \ref{thm:holeprobability1} to the case of infinite Ginibre ensemble.

 \begin{theorem}\label{thm:infinite}
Let $\X_\infty$ be the point process of infinite Ginibre ensemble and let $U$ be an open subset of $\mathbb D$ which satisfies the conditions in the hypothesis of Theorem \ref{thm:holeprobability12} or that of Theorem  \ref{thm:holeprobability1}. Then
$$
\lim_{n\to \infty}\frac{1}{r^4}\log\P[\X_{\infty}(rU)=0]=R_{\emptyset}-R_U,
$$
where  $\X_{\infty}(U)$ denotes the number of points of $\X_{\infty}$ that fall in $U$.
\end{theorem}

The paper is organized as follows. In Section \ref{results}, we give exact values of the constant $R_U$ and the balayage measure $\nu_2$ for some particular open sets $U$.  Assuming  Theorem \ref{thm:generalformula}, Theorem \ref{thm:holeprobability12} and Theorem \ref{thm:holeprobability1},  we give proof of Theorem \ref{thm:infinite} in Section \ref{proov}.  In Section \ref{proof}, we give proof of Theorem \ref{thm:generalformula}. In Section \ref{gap}, we give proofs of Theorems \ref{thm:holeprobability12} and \ref{thm:holeprobability1}. We show  explicit calculations for finding the constant $R_U$ and the balayage measure $\nu_2$ in Section \ref{example}. In Section \ref{sec:annulus}, we prove Theorem \ref{thm:annulus1}.

\section{Table of examples}\label{results}
 Suppose $\nu $ is the
equilibrium measure for $\overline{\mathbb D}\backslash U$ as in Theorem \ref{thm:generalformula}. Then the equilibrium measure is
$\nu=\nu_1+\nu_2$, where
 $$
d\nu_1(z)=\left\{\begin{array}{lr}\frac{1}{\pi}dm(z) & \mbox{if
$z\in  \overline{ \mathbb D} \backslash {U}$}\\0& \mbox{o.w.}\end{array}\right.,
 $$
 and $\nu_2$
is the balayage measure on $\partial U$. Let $R_U'=\frac{1}{2}\left[\int_{\partial U}|z|^2d\nu_2(z)-\frac{1}{\pi}\int_{{U}}|z|^2dm(z)\right]$, then $R_U=\frac{3}{4}+R_U'$. Note that $R_U'$ satisfies the scaling relation, $R_{aU}'=a^4R_U'$.
The balayage measure $\nu_2$ and $R_U'$, for some particular
open sets $U$, are given in the following table.

\begin{center}
\begin{tabular}{ | m{3.5cm} | m{6.5cm}| m{2.6cm} | }
\hline
\begin{center}{\bf $U$}\end{center} & \begin{center}{\bf $\nu_2$}\end{center} & \begin{center}{\bf $R_{U}'$}\end{center} \\
 \hline	
 $\{z:|z|<a\}$ (disk).&
$d\nu_2(z)=\left\{\begin{array}{lr}\frac{a^2}{2\pi}d\theta &
\mbox{if $z=ae^{i\theta}$}\\0& \mbox{o.w.}\end{array}\right.$
& \begin{center}$\frac{a^4}{4}$\end{center}	\\
\hline
$\{z:|z-a_0|<a\}$ \mbox{for fixed $a_0\in \mathbb D .$}
&$d\nu_2(z)=\left\{\begin{array}{lr}\frac{a^2}{2\pi}d\theta &
\mbox{if $z=a_0+ae^{i\theta}$}\\0& \mbox{o.w.}\end{array}\right.$
& \begin{center}$\frac{a^4}{4}$\end{center}\\
\hline
$\{z:a<|z|<b\}$ for $0<a<b<1,$ (annulus).&
 $ d\nu_2'(z) =\left\{\begin{array}{lr}
\lambda(b^2-a^2)\frac{d\theta}{2\pi}&\mbox{if } z=ae^{i\theta}\\0&
\mbox{o.w.}
\end{array}\right.$,
$d\nu_2''(z)$ $= \left\{\begin{array}{lr}
(1-\lambda)(b^2-a^2)\frac{d\theta}{2\pi}&\mbox{if } z=be^{i\theta}\\0&
\mbox{o.w.}
\end{array}\right.$ where
$\lambda = \frac{(b^2-a^2)-2a^2\log(b/a)}{2(b^2-a^2)\log(b/a)}$
and $\nu_2=\nu_2'+\nu_2''$.
& $\frac{1}{4}(b^4-a^4)$ $-\frac{1}{4}\frac{(b^2-a^2)^2}{\log(b/a)}$ \\
\hline	
$\{(x,y) |
\frac{x^2}{a^2}+\frac{y^2}{b^2}< 1\}$ (ellipse).
&$d\nu_2(z)=\frac{ab}{2\pi}\left[1-{\frac{a^2-b^2}{a^2+b^2}}\cos(2\theta)\right]d\theta$
when $z\in \partial U$&
$\frac{1}{2}\cdot\frac{(ab)^3}{a^2+b^2}$ \\
\hline
$\{re^{i\theta}|0\leq r< b(1+2a\cos\theta), 0\le \theta\le
2\pi\}$ (Cardioid). & $d\nu_2(z)=\frac{b^2}{2\pi}(1+a^2+2a\cos
\theta)d\theta
$ when $z\in \partial U$. & $\frac{b^4}{2}(a^2+1)^2-\frac{b^4}{4}$\\
\hline
Fix $a<1$. $aT$ where $T$ be triangle with cube roots of
unity $1,\omega,\omega^2$ as vertices. &
\begin{center}$\ldots$\end{center}
&$\frac{a^4}{2\pi}\cdot \frac{9\sqrt{3}}{80}$\\
\hline
$\{re^{i\theta}:0<r<a,0<\theta<\pi\}$, (half-disk).
&\begin{center}$\ldots$\end{center}  &
$\frac{a^4}{2}\left(\frac{1}{2}-\frac{4}{\pi^2}\right)$\\
  \hline  
\end{tabular}
\end{center}
In all the above  examples, the parameters $a,b$  are such that $U \subseteq D$.  For detailed calculations, see Section \ref{example}.

\section{proof of Theorem \ref{thm:infinite}}\label{proov}

We assume Theorem \ref{thm:generalformula}, Theorem \ref{thm:holeprobability12} and Theorem \ref{thm:holeprobability1} in order to prove Theorem \ref{thm:infinite}. 

\begin{proof}[Proof of Theorem \ref{thm:infinite}]
Since the $n$-th Ginibre ensemble converges in distribution to infinite Ginibre ensemble  as $n \to \infty$, we have that 
$$
\P[\X_{\infty}(rU)=0]=\lim_{n \to \infty} \P [\X_{n}(rU)=0].
$$
Using the determinantal nature of density of eigenvalues of $G_n$, we get 
\begin{align*}
&\P[\X_n(rU)=0]
\\&=\frac{1}{n!}\int_{(rU)^c}\cdots \int_{(rU)^c}\det (K_n(z_i,z_j))_{1\le i,j\le n}\prod_{i=1}^ndm(z_i)
\\&=\frac{1}{n!}\int_{(rU)^c}\cdots \int_{(rU)^c}\det (\varphi_k(z_i))\det (\overline{\varphi_k(z_i)})\prod_{i=1}^ndm(z_i)
\\&=\frac{1}{n!}\int_{(rU)^c}\cdots \int_{(rU)^c}\sum_{\sigma,\tau \in S_n}{\mbox{sgn}(\sigma\tau)}\prod_{i=1}^{n}\varphi_{\sigma(i)}(z_i)\overline {\varphi_{\tau(i)}(z_i)}\prod_{i=1}^n dm(z_i)
\\&=\sum_{\sigma\in S_n}{\mbox{sgn}(\sigma)}\prod_{i=1}^{n}\int_{(rU)^c}\varphi_{i}(z)\overline {\varphi_{\sigma(i)}(z)}dm(z)
\\&=\det\left(\int_{(rU)^c}\varphi_{i}(z)\overline {\varphi_{j}(z)}dm(z)\right)_{1\le i,j\le n}.
\end{align*}
Let us define
 $$M_n(rU):=\left(\int_{(rU)^c}\varphi_{i}(z)\overline {\varphi_{j}(z)}dm(z)\right)_{1\le i,j\le n}.$$  $M_n(rU)$ is the integral of the positive definite matrix function $\left(\varphi_{i}(z)\overline {\varphi_{j}(z)}\right)_{1\le i,j\le n}$ over the region $(rU)^c$. So, we have  that 
$ M_n(rU)\ge M_n(r\D)\ge 0$ for all $n$ and $U\subseteq \D$. Observe that for a positive definite matrix $\begin{bmatrix}A & B \\ C & D\end{bmatrix}$, we have 
$$
\det\begin{bmatrix}A & B \\ C & D\end{bmatrix}\le \det(A)\det(D). 
$$
 Therefore we have
$$
\det (M_n(rU))\le \det (M_{n-1}(rU)).\int_{(rU)^c}\varphi_n(z)\overline{\varphi_n(z)}dm(z)\le \det (M_{n-1}(rU)).
$$
Therefore $\P[\X_n(rU)=0]=\det (M_n(rU))$ decreases with $n$ and converges to $\P[\X_{\infty}(rU)=0]$. Therefore,  for all $n\ge 2r^2$, we have
\begin{align}\label{eqn:up}
\P[\X_{2r^2}(rU)=0]\ge \P[\X_n(rU)=0]\ge \P[\X_{\infty}(rU)=0].
\end{align}
Again for $n> 2r^2$, we have 
\begin{align}\label{eqn:det}
\P[\X_n(rU)=0]= \det(M_n(rU))=\det(M_{2r^2}(rU))\det([M_n(rU)/M_{2r^2}(rU)]),
\end{align}
where $[M_n(rU)/M_{2r^2}(rU)]$ is  the Schur complement of the block $M_{2r^2}(rU)$ of the matrix $M_{n}(rU)$. Since $M_n(rU)\ge M_n(r\D)\ge 0$, the Schur complements satisfy the inequality
$$
[M_n(rU)/M_{2r^2}(rU)]\ge [M_n(r\D)/M_{2r^2}(r\D)].
$$
By min-max theorem, we have
$$
\det([M_n(rU)/M_{2r^2}(rU)])\ge \det([M_n(r\D)/M_{2r^2}(r\D)]).
$$
As $\D$ is circular domain,  we have 
$$
\int_{(r\D)^c}\varphi_i(z)\overline{\varphi_j(z)}dm(z)=0\;\;\; \mbox{for all $i\neq j$}.
$$
Therefore $M_n(r\D)=\mbox{diag}\left(\int_{(r\D)^c}|\varphi_1(z)|^2dm(z),\ldots,\int_{(r\D)^c}|\varphi_n(z)|^2dm(z)\right)$. For large $r$, we have
\begin{eqnarray}\label{eqn:tailed}
\det([M_n(rU)/M_{2r^2}(rU)])&\ge&\prod_{k=2r^2+1}^{n}\int_{(r\D)^c}|\varphi_k(z)|^2dm(z)\nonumber
\\&\ge& \prod_{k=2r^2+1}^{\infty}\int_{(r\D)^c}|\varphi_k(z)|^2dm(z)\ge C \;\;\; \mbox{(by \eqref{eqn:tail})},
\end{eqnarray}
for some positive constant $C$. By \eqref{eqn:det} and \eqref{eqn:tailed}, for large $r$, we get 
\begin{eqnarray}\label{eqn:low}
\P[\X_{\infty}(rU)=0]&=&\lim_{n\to \infty}\P[\X_n(rU)=0]\ge C.\P[\X_{2r^2}(rU)=0].
\end{eqnarray}
Therefore from \eqref{eqn:up} and \eqref{eqn:low} we get 
\begin{eqnarray*}
\lim_{r\to \infty}\frac{1}{r^4}\log\P[\X_{\infty}(rU)=0]&=&\lim_{r\to \infty}\frac{1}{r^4}\log\P[\X_{2r^2}(rU)=0]\\&=&\lim_{n\to \infty}\frac{4}{n}\log\P[\X_{n}(\sqrt{n}\cdot\frac{U}{\sqrt{2}})=0].
\end{eqnarray*}
Since $U$ satisfies the conditions of either  Theorem \ref{thm:holeprobability12} or Theorem \ref{thm:holeprobability1}, we have 
\begin{eqnarray*}
\lim_{r\to \infty}\frac{1}{r^4}\log\P[\X_{\infty}(rU)=0]&=&4.\left(-R_{\frac{U}{\sqrt 2}}+\frac{3}{4}\right)=-4.R_{\frac{U}{\sqrt 2}}'
\\&=&-R_{{U}}'=-R_{{U}}+\frac{3}{4},
\end{eqnarray*}
third equality follows from the Theorem \ref{thm:generalformula}.
\end{proof}

\section{proof of Theorem \ref{thm:generalformula}}\label{proof}

In this section we present the proof of Theorem \ref{thm:generalformula}. Before proceeding to the proof, we provide some basic definitions and facts of classical potential theory from \cite{ransfordbook}, \cite{totikbook}.

Support of a positive measure $\mu$ on $\mathbb{C}$, denoted by $supp(\mu)$, consists of all points $z$ such that $\mu(D_r(z))>0$ for every open disk $D_r(z)$ of radius $r>0$ and with center at $z$. The measure $\mu$ is said to be compactly supported measure if $supp(\mu)$ is compact. Let $\mu$ be a compactly supported probability measure on $\mathbb C$. Then its {\it potential} is the function $p_{\mu}:\mathbb C \to (-\infty, \infty]$ defined by
$$
p_{\mu}(z):=-\int \log |z-w|d\mu(w)\;\;\;\mbox{for all $z\in \mathbb C$}.
$$
Its {\it logarithmic energy}  $I_\mu$ is defined by
$$
I_\mu:=-\iint  \log|z-w|d\mu(z)d\mu(w)=\int p_{\mu}(z)d\mu(z).
$$

A set $E \subset \mathbb{C}$ is said to be polar if $I_{\mu}=\infty$ for all compactly supported probability measures $\mu$ with $supp({\mu}) \subset E$. The {\it capacity} of a subset $E$ of $\mathbb C$ is given by
$$
C(E):=e^{-\inf\{I_{\mu}:\mu \in \mathcal P(E)\}}.
$$
Clearly, the capacity of the polar sets is zero. The singleton sets are polar  and countable union of polar sets  is again a polar set. 
A property is said to  hold {\it quasi-everywhere} (q.e.) on $E\subset \mathbb C$ if it holds everywhere on $E$ except some borel polar set. Every Borel probability measure with finite logarithmic  energy assigns zero measure to Borel polar sets (see Theorem 3.2.3, \cite{ransfordbook}). So, a property, which holds q.e. on $E$, holds {\it $\mu$-everywhere} on $E$, for every $\mu$ with finite energy. As a corollary, we have that every Borel polar set has Lebesgue measure zero and   a property, which holds q.e. on $E$, holds a.e. on $E$.

 A weight function $w:E\to [0,\infty)$, on a closed subset $E$ of $\mathbb{C}$, is said to be {\it
admissible } if it satisfies the following three conditions:
\begin{enumerate}
\item $w$ is upper semi-continuous,

\item $E_0:=\{z\in E|w(z)>0\}$ has positive capacity,
\item if $E$ is unbounded, then $|z|w(z)\to 0$ has
$|z|\to\infty,z\in E$.
\end{enumerate}

\noindent{\bf Equilibrium measure: }The probability measure, with support in $E$,  which minimizes 
$$
R_{\mu}=\int p_{\mu}(z)d\mu(z)+2\int Q(z)d\mu(z),
$$ where $w=e^{-Q}$ is an admissible weight function, is called {\it equilibrium measure} for $E$ with external field $Q$. We have the following facts regarding equilibrium measure.

\begin{fact}\label{ft:characterization}
Let $w=e^{-Q}$ be an admissible weight function on closed set $E$. Then there exists a unique equilibrium measure $\nu$, for $E$ with external field $Q$. The equilibrium measure $\nu$  has compact support and $R_{\nu}$ is finite (so is $I_{\nu}$). $\nu$ satisfies the following conditions 
\begin{eqnarray}\label{eqn:insupport}
 p_{\nu}(z)+Q(z)=C
\end{eqnarray}
for q.e. $z\in supp(\nu)$ and
\begin{eqnarray}\label{eq:outsidesupport}
 p_{\nu}(z)+Q(z)\ge C
\end{eqnarray}
for q.e. $z\in E$ for some constant $C$. Also, the above conditions uniquely characterize the equilibrium measure, i.e. a probability measure with compact support in $E$ and finite energy, which satisfies the conditions \eqref{eqn:insupport} and \eqref{eq:outsidesupport} for some constant $C$, is the equilibrium measure for $E$ with external field $Q$. 
\end{fact}
For the proof of this fact, see (\cite{totikbook}, Chapter I Theorem 1.3 and Theorem 3.3).
The discrete analogue of the above minimization problem of $R_{\mu}$ is the problem of finding the limit of 
$$
\delta_n^{\omega}(E):=\sup_{z_1,z_2,\ldots,z_n\in E}\left\{
\prod_{i<j}|z_i-z_j|\omega(z_i)\omega(z_j)\right\}^{\frac{2}{n(n-1)}},
$$
as $n \to \infty$. The sets $\mathcal{F}_n=\{z_1^*,z_2^*,\ldots,z_n^* \} \subset E$ are said to be {\it $n$-th weighted Fekete sets} for $E$ 
if 
$$
\delta_n^{\omega}(E)=\left\{
\prod_{i<j}|z_i^*-z_j^*|\omega(z_i^*)\omega(z_j^*)\right\}^{\frac{2}{n(n-1)}}.
$$
The points $z_1^*,z_2^*,\ldots,z_n^*$ in weighted Fekete set $\mathcal{F}_n$ are called $n$-th  weighted Fekete points.  It is known that the sequence $\{\delta_n^{\omega}(E)\}_{n=2}^{\infty}$ decreases to $e^{-R_{\nu}}$, i.e. 
\begin{eqnarray}\label{eqn:limit}
\lim_{n \to \infty} \delta_n^{\omega}(E)= e^{-R_{\nu}}.
\end{eqnarray}
Moreover, the uniform probability measures on $n$-th weighted Fekete sets converge weakly to equilibrium measure $\nu$, i.e.
$$
\lim_{n \to \infty} {\nu}_{\mathcal{F}_n}= {\nu},
$$
where ${\nu}_{\mathcal{F}_n}$ is uniform measure on ${\mathcal{F}_n}$. For the proofs of these facts, see \cite{totikbook}, Chapter III Theorem 1.1 and Theorem 1.3. The following fact  (an application of Theorem 4.7 in Chapter II, \cite{totikbook},  to bounded open sets) is about  the existence and uniqueness of the balayage measure.  
\begin{fact}\label{ft:balayage}
Let $U$ be an bounded open subset of $\C$ and $\mu$ be finite Borel measure on $U$ (i.e., $\mu(U^c)=0$). Then there exists a unique measure $\hat \mu $ on $\partial U$ such that $\hat \mu(\partial U)=\mu(U)$, $\hat \mu(B)=0$ for every Borel polar set $B\subset \C$ and $p_{\hat \mu(z)}=p_{\mu}(z)$ for q.e. $z\in U^c$. $\hat \mu$ is said to be the balayage measure associated with $\mu$ on $U$.  
\end{fact}

 As we are  are interested in the hole probabilities for Ginibre ensemble, we take for $E$  sets of the form $\mathbb{C}\backslash U$, for some open $U \subseteq \mathbb{D}$, and $Q(z)=\frac{|z|^2}{2}$. Clearly, $w(z)=e^{\frac{1}{2}|z|^2}$ is an admissible weight function on such $E$.

 First we give   basic examples of equilibrium measures,
then we prove Theorem \ref{thm:generalformula}.  Recall the notation, $R_U =\inf\{R_{\mu}:\mu \in \mathcal P(\mathbb C\backslash U)\}$, that we are using to emphasize the  hole probabilities. We use the following well known fact, known as Jensen's formula.
\begin{fact}\label{ft:fundamental}
For each $r>0$,
$$
\frac{1}{2\pi}\int_{0}^{2\pi}\log\frac{1}{|z-re^{i\theta}|}d\theta=\left\{\begin{array}{lr}\log
\frac{1}{r} & \mbox{ if } |z|\le r\\\log \frac{1}{|z|} & \mbox{ if
} |z|> r\end{array}\right..
$$
\end{fact}
\noindent To see this, note that $\log |1-z|$ is harmonic on $\D$, by mean value property we have 
$$
\frac{1}{2\pi}\int_{0}^{2\pi}\log|1-re^{i\theta}|d\theta=0 \;\;\mbox{for $r<1$}.
$$
This equality is  true also for $r=1$, by direct calculation. This implies that $\frac{1}{2\pi}\int_{0}^{2\pi}\log|z-re^{i\theta}|d\theta=\log(\max\{|z|,r\})$.

\begin{fact}\label{thm:disk}
 Suppose $E=\mathbb C$ and $Q(z)=\frac{|z|^2}{2}$ on $\mathbb C$.
Then the equilibrium measure $\mu$ is normalized uniform measure on unit disk and $R_{\emptyset}=\frac{3}{4}$, i.e.,
$$
d\mu(z)=\left\{\begin{array}{lr} \frac{1}{\pi}dm(z)& \mbox{if  }
z\in \overline{\mathbb{D}}\\0& \mbox{otherwise}
\end{array}\right. \mbox{ and } R_{\mu}=\frac{3}{4},
$$
where $\mathbb D$ denotes the open unit disk.
\end{fact}

 \begin{proof}[Proof of Fact \ref{thm:disk}]
 Let $d\mu(z)=\frac{1}{\pi}dm(z)$ when $z\in \overline{\mathbb D}$ and zero other wise. We show that the measure $\mu$ satisfies the conditions \eqref{eqn:insupport} and \eqref{eq:outsidesupport}.
 Hence by Fact \ref{ft:characterization} we conclude that $\mu$ is the equilibrium measure.

\noindent Let $z\in \overline{\mathbb D}$. Then we have
\begin{eqnarray}\label{eqn:equality}
p_{\mu}(z)&=&\frac{1}{\pi}\int_{\overline{\mathbb D}} \log\frac{1}{|z-w|}dm(w)\nonumber
\\&=&\frac{1}{\pi}\int_0^1\int_{0}^{2\pi} \log\frac{1}{|z-re^{i\theta}|}rdrd\theta\nonumber
\\&=&2\log\frac{1}{|z|}\int_0^{|z|}rdr+2\int_{|z|}^1\log\frac{1}{r}rdr\;\;\mbox{ (by Fact \ref{ft:fundamental})}\nonumber
\\&=&\frac{1}{2}-\frac{|z|^2}{2}.
\end{eqnarray}
Which implies that $\mu$ satisfies the condition
\eqref{eqn:insupport}. On other hand for $|z|\ge 1$ we have
\begin{eqnarray*}
p_{\mu}(z)&=&\log\frac{1}{|z|}.
\end{eqnarray*}
Since the function $f(r):=\frac{r^2}{2}+\log\frac{1}{r}$ is
strictly increasing for $r>1$ and $f(1)=\frac{1}{2}$. Hence $\nu$ satisfies
the condition \eqref{eq:outsidesupport}, i.e.,
\begin{eqnarray}\label{eqn:thm1}
p_{\mu}(z)+\frac{|z|^2}{2}\ge \frac{1}{2}
\end{eqnarray}
for $|z|\ge 1$. Therefore $\mu$ is the equilibrium measure for $\overline{\mathbb D}$ with external field $\frac{|z|^2}{2}.$

\noindent {\bf Value of $R_{\emptyset}$: } We have
\begin{eqnarray}\label{eqn:part1}
\int_{\overline{\mathbb D}} |z|^2d\mu(z)=\frac{1}{\pi}\int_{0}^1\int_{0}^{\pi}r^2rdr d\theta
=2\int_{0}^{1}r^3dr=\frac{1}{2}.
\end{eqnarray}

\noindent By Fact \ref{ft:fundamental}, we have 
\begin{eqnarray}\label{eqn:part2}
\int_{\overline{\mathbb D}}p_{\mu}(z) d\mu(z) 
=\frac{1}{4}.
\end{eqnarray}
Therefore by (\ref{eqn:part1}) and (\ref{eqn:part2}) we get  $R_{\emptyset}=\frac{3}{4}$.
\end{proof}

\begin{remark}\label{re:disk}
If $E=\C \backslash \D$ and $Q(z)=\frac{|z|^2}{2}$ on $\C$. Then by similar calculation it can be shown that the equilibrium measure $\nu$ is uniform measure on unit circle, i.e.
$$
d\nu(z)=\left\{\begin{array}{lr}\frac{1}{2\pi}d\theta &  \mbox{if }  |z|=1\\ 0 & \mbox{otherwise}\end{array}\right. \mbox{ and } R_{\mathbb{D}}=R_{\nu}=1.
$$

\end{remark}

Now, we proceed to prove Theorem \ref{thm:generalformula}.

\begin{proof}[Proof of Theorem \ref{thm:generalformula}]
Let $U\subseteq \D $ be an open set. Set $E=U^c$, $Q(z)=\frac{|z^2|}{2}$. Clearly, $w=e^{-Q}$ is an  admissible weight function on $E$. Let $\mu$
be the uniform probability measure on unit disk. Let $\mu=\mu_1+\mu_2$, where $\mu_1$ and $\mu_2$ are $\mu$ restricted to $U^c$ and $U$ respectively. By Fact \ref{ft:balayage}, we know that there exists a measure $\nu_2$ on $\partial U$ such that $\nu_2(\partial U)=\mu_2(U)$, $\nu_2(B)=0$ for every Borel polar set and 
$$
p_{\nu_2}(z)=p_{\mu_2}(z)\;\; \mbox{ q.e. on $U^c$}.
$$ 
Define $\nu=\mu_1+\nu_2$. Then  we have that the support of $\nu$ is $\overline \D\backslash U$ and 
$$
p_{\nu}(z)=p_{\mu_1}(z)+p_{\nu_2}(z)=p_{\mu_1}(z)+p_{\mu_2}(z)=p_{\mu}(z)\;\; \mbox{ q.e. on $U^c$}.
$$
 From  \eqref{eqn:equality} and \eqref{eqn:thm1}, we have 
\begin{eqnarray*}
p_{\mu}(z)+\frac{|z|^2}{2}&=&\frac{1}{2}\;\;\mbox{ for $z\in \overline \D$}
\\ p_{\mu}(z)+\frac{|z|^2}{2}&\ge&\frac{1}{2}\;\;\mbox{ for $z\in \overline \D^c$}.
\end{eqnarray*}
This gives us that 
\begin{eqnarray*}
p_{\nu}(z)+\frac{|z|^2}{2}&=&\frac{1}{2}\;\;\mbox{ for q.e. $z\in supp(\nu) $}
\\ p_{\nu}(z)+\frac{|z|^2}{2}&\ge&\frac{1}{2}\;\;\mbox{ for q.e. $z\in  E$}.
\end{eqnarray*}
The energy of the measure  $\nu$,
\begin{eqnarray*}
I_{\nu}&=&\int p_{\nu}(z)d\nu(z)
=\int \left(\frac{1}{2}-\frac{|z|^2}{2}\right)d\nu(z),
\end{eqnarray*}
is finite. The second equality follows from the fact that  $\nu(B)=0$ for all Borel polar sets $B$. So, $\nu $ has finite energy and satisfies conditions \eqref{eqn:insupport} and \eqref{eq:outsidesupport}. Therefore,  by Fact \ref{ft:characterization}, $\nu$ is the equilibrium measure for $U^c$ with the quadratic external field $\frac{|z|^2}{2}$.

\vspace{.5cm} \noindent{\bf Value of $R_U$ :} 
We have
$$
 p_{\nu}(z)=\frac{1}{2}-\frac{|z|^2}{2},
$$
q.e on the support of $\nu$. Therefore we have
\begin{eqnarray*}
R_{\nu}&=&\int p_{\nu}(z)d\nu(z)+\int |z|^2d\nu(z)
\\&=&\frac{1}{2}-\int \frac{|z|^2}{2}d\nu(z)+\int |z|^2d\nu(z)
\\&=&\frac{1}{2}+\int \frac{|z|^2}{2}d\mu_1(z)+\int \frac{|z|^2}{2}d\nu_2(z)
\\&=&\frac{3}{4}+\frac{1}{2}\left[\int |z|^2d\nu_2(z)-\int |z|^2d\mu_2(z)\right].
\end{eqnarray*}
Last equality follows from the fact that $\int |z|^2d\mu_1(z) + \int |z|^2d\mu_2(z)=\frac{1}{2} $.
 Therefore we get
\begin{eqnarray}\label{eqn:robinconstant}
R_{\nu}=\frac{3}{4}+\frac{1}{2}\left[\int_{\partial U} |z|^2d\nu_2(z)-\frac{1}{\pi}\int_{{U}} |z|^2dm(z)\right].
\end{eqnarray}
The result follows from the fact that $R_U=R_{\nu}$.
\end{proof}

\begin{remark}\label{relations}
Let $\nu_2$ and $\mu_2$ be as in the above proof i.e. $\nu_2$ is the balayage measure associated with $\mu_2$.  We have $p_{\nu_2}(z)=p_{\mu_2}(z)$ for q.e.  $z\in { U}^c$.  As the logarithmic potential of a measure is harmonic outside its support, the above relation holds for every $z \in {\overline U}^c$. Outside $\D$,  $p_{\nu_2}(z)$ and $p_{\mu_2}(z)$ are real parts of the analytic functions $ -\int_{\partial U}\log{(z-w)}d\nu_2(w)$ and $ -\int_{\partial U}\log{(z-w)}d\mu_2(w)$, respectively. 
So  there exists a constant $c$ such that for all $|z|>1$,
\begin{eqnarray}\label{eqn:moment}
\int_{\partial U}\log{(z-w)}d\nu_2(w)&=&\int_{ U}\log{(z-w)}d\mu_2(w) +c,\nonumber\\
 \Leftrightarrow \int_{\partial U}[\log{z}+\sum_{n=1}^{\infty}
\frac{w^n}{nz^n}]d\nu_2(w)&=&\int_{ U}[\log{z}+\sum_{n=1}^\infty
\frac{w^n}{nz^n}]d\mu_2(w)+c\nonumber\\
\Leftrightarrow \int_{\partial U}w^n d\nu_2(w)&=&\int_{ U}w^nd\mu_2(w), \forall
n\geqslant0 \mbox{  and $c=0$}\nonumber\\
\Leftrightarrow \int_{\partial U}w^n d\nu_2(w)&=&\frac{1}{\pi}\int_{ U}w^ndm(w), \forall
n\geqslant0. 
\end{eqnarray}
To see the converse of the above, suppose $\nu_2$ is a measure on $\partial U$ which satisfies  the relations \eqref{eqn:moment}, then $p_{\nu_2}(z)=p_{\mu_2}(z)$ for every   $z\in { \overline U}^c.$ If $\partial U$ is a piecewise smooth curve and $\nu_2$ has  density with respect to arc-length on $\partial U$, then $p_{\nu_2}(z)$ is continuous  at all the continuity points of the density of $\nu_2$ (\cite{totikbook} Chapter II Theorem 1.5). In this case $p_{\mu_2}(z)$ is also continuous on $\partial U$. So if the density of $\nu_2$ is piecewise continuous   on $\partial U$, we get that 
$p_{\nu_2}(z)=p_{\mu_2}(z)$ for q.e.   $z\in {  U}^c.$  Therefore when $\partial U$ is piecewise smooth curve,  a measure $\nu_2$ on $\partial U$ which has piecewise continuous density with respect to arclength and satisfies relations  \eqref{eqn:moment} is the balayage measure on $\partial U$.  
\end{remark}

\section{Proofs of Theorems \ref{thm:holeprobability12} and \ref{thm:holeprobability1}}\label{gap}
In this section, we give the proofs of Theorem
\ref{thm:holeprobability12} and Theorem
\ref{thm:holeprobability1}.  First, we give an  upper bound for hole probabilities.

\noindent{\bf Upper bound: } From \eqref{eqn:holeprob} we have
\begin{eqnarray}\label{eqn:feket}
\P[\X_n(\sqrt n U)=0]&=&\frac{1}{Z_n}\int_{U^c}\ldots \int_{U^c}
e^{-n\sum_{k=1}^n|z_k|^2}\prod_{i<j}|z_i-z_j|^2\prod_{k=1}^n
dm(z_k)\nonumber
\\&=&\frac{1}{Z_n}\int_{U^c}\ldots \int_{U^c}\left\{ \prod_{i<j}|z_i-z_j|\omega(z_i)\omega(z_j)\right\}^2\prod_{k=1}^ne^{-|z_k|^2} dm(z_k),
\end{eqnarray}
where $\omega(z)=e^{-\frac{|z|^2}{2}}$. Let
$z_1^*,z_2^*,\ldots,z_n^*$ be weighted Fekete points for $U^c$.
Therefore we have
$$
\delta_n^{\omega}(U^c)=\left\{
\prod_{i<j}|z_i^*-z_j^*|\omega(z_i^*)\omega(z_j^*)\right\}^{\frac{2}{n(n-1)}}.
$$
Therefore from \eqref{eqn:feket} we have
\begin{eqnarray*}
\P[\X_n(\sqrt n U)=0]&\le
&\frac{1}{Z_n}(\delta_n^{\omega}(U^c))^{n(n-1)}\prod_{k=1}^n\left(\int_{U^c}e^{-|z_k|^2}dm(z_k)\right)
\\&\le &\frac{1}{Z_n}.a^n.(\delta_n^{\omega}(U^c))^{n(n-1)} ,
\end{eqnarray*}
where $a=\int_{U^c}e^{-|z|^2}dm(z)$. Again we have 
\begin{eqnarray}\label{re:thecostant}
\lim_{n\to \infty}\frac{\log Z_n}{n^2}&=&\lim_{n\to \infty}\frac{1}{n^2}\left(-\frac{n^2}{2}\log n+n\log \pi +\sum_{k=1}^n\sum_{r=1}^k \log r \right)\nonumber
\\&=&\lim_{n\to \infty}\frac{1}{n^2}\left(\sum_{k=1}^n\sum_{r=1}^k \log \frac{r}{n} \right)\nonumber
\\&=&\int_{0}^1\int_{0}^y\log x dxdy=-\frac{3}{4}.
\end{eqnarray}
Therefore by \eqref{eqn:limit} and \eqref{re:thecostant}, we have
\begin{equation}\label{inequality}
\limsup_{n\to \infty}\frac{1}{n^2}\log \P[\X_n(\sqrt n U)=0]\le -\inf_{\mu\in \mathcal {P}(\C \backslash U)}R_{\mu}+\frac{3}{4} = -R_U+\frac{3}{4}.
\end{equation}
 Note that this upper bound is true for any open $U$.

 The following lemma, which is used in the proof of Theorem \ref{thm:holeprobability1}, provides separation between weighted Fekete points which are being considered in this paper. The separation of Fekete points has been studied by many authors, e.g., \cite{ortega}, \cite{bos}.

\begin{lemma}\label{lem:feketedistance}
Let $U$ be an open subset of $\mathbb D$ satisfying  condition
\eqref{eqn:condition} and $z_1^*,z_2^*,\ldots,z_n^*$ be weighted
Fekete points for $U^c$. Then, for large n,
\begin{eqnarray*}\label{eqn:feketedistance}
\min\{|z_i^*-z_k^*|: 1\le i\neq k \le n\}\ge C.\frac{1}{n^3}
\end{eqnarray*}
for some constant $C$ (does not depend on $n$).
\end{lemma}

\noindent Assuming Lemma \ref{lem:feketedistance} we proceed to
prove  Theorem \ref{thm:holeprobability1}.
\begin{proof}[Proof of Theorem \ref{thm:holeprobability1}]
 Let
$z_1^*,z_2^*,\ldots,z_n^*$ be weighted Fekete points for $U^c$. It
is known that $|z_{\ell}^*|\le 1$ for $\ell=1,2,\ldots, n$ (see
\cite{totikbook}, Chapter III Theorem 2.8). Therefore we have
\begin{eqnarray*}\label{eqn:lower}
\P[\X_n(\sqrt n U)=0]\nonumber &=&\frac{1}{Z_n}\int_{U^c}\ldots
\int_{U^c}
e^{-n\sum_{k=1}^n|z_k|^2}\prod_{i<j}|z_i-z_j|^2\prod_{k=1}^n
dm(z_k)\nonumber
\\&\ge &\frac{1}{Z_n}\int_{B_1}\ldots \int_{B_n}
\left\{
\prod_{i<j}|z_i-z_j|\omega(z_i)\omega(z_j)\right\}^2\prod_{k=1}^ne^{-|z_k|^2}
dm(z_k),\nonumber
\\&\ge &\frac{e^{-2n}}{Z_n}\int_{B_1}\ldots \int_{B_n}
\left\{ \prod_{i<j}|z_i-z_j|\omega(z_i)\omega(z_j)\right\}^2\prod_{k=1}^n dm(z_k),
\end{eqnarray*}
where $B_\ell=U^c\cap B(z_{\ell}^*,\frac{C}{n^4})$ for
$\ell=1,2,\ldots, n$ and $\omega(z)=e^{-\frac{|z|^2}{2}}$. By
Lemma \ref{lem:feketedistance} we have $|z_i^*-z_j^*|\ge
\frac{C}{n^3}$ for $i\neq j$, for some constant $C$ independent of
$n$. Suppose $z_i\in B(z_i^*,\frac{C}{n^4})$ and $z_j\in
B(z_j^*,\frac{C}{n^4})$ for $i\neq j$, then  we have
\begin{eqnarray*}
|z_i-z_j|\ge |z_i^*-z_j^*|-\frac{2C}{n^4}\ge
|z_i^*-z_j^*|-\frac{2}{n}\cdot |z_i^*-z_j^*| \ge
|z_i^*-z_j^*|\left(1-\frac{2}{n}\right).
\end{eqnarray*}
Therefore we have
\begin{eqnarray*}
&&\P[\X_n(\sqrt n U)=0] \\&\ge& \frac{e^{-2n}}{Z_n}\int_{B_1}\ldots
\int_{B_n}\left\{
\prod_{i<j}|z_i^*-z_j^*|\left(1-\frac{2}{n}\right)\omega(z_i)\omega(z_j)\right\}^2
\prod_{k=1}^n dm(z_k)
\\&\ge &  \frac{e^{-2n}}{Z_n}\left(1-\frac{2}{n}\right)^{{n(n-1)}}.e^{-\frac{2C}{n^2}}.\left\{
\prod_{i<j}|z_i^*-z_j^*|\omega(z_i^*)\omega(z_j^*)\right\}^2\prod_{k=1}^n\int_{B_k}
dm(z_k)
\end{eqnarray*}
The last inequality follows from the fact that
$$
e^{-\frac{1}{2}|z_i|^2}\ge
e^{-\frac{1}{2}(|z_i^*|+\frac{C}{n^4})^2}\ge
e^{-\frac{1}{2}|z_i^*|^2}.e^{-\frac{C}{n^4}(|z_i^*|+\frac{C}{2n^4})}\ge
e^{-\frac{1}{2}|z_i^*|^2}.e^{-\frac{2C}{n^4}},
$$
for $z_i\in B(z_i^*,\frac{C}{n^4}), i=1,2,\ldots,n$. For large $n$, we have
$\int_{B_i} dm(z_i)\ge \pi.(\frac{ C}{2n^4})^2,
i=1,2,\ldots,n$ (by condition \eqref{eqn:condition}). Hence we
have
$$
\P[\X_n(\sqrt n U)=0]\ge
\frac{e^{-2n}}{Z_n}\left(1-\frac{2}{n}\right)^{{n(n-1)}}.e^{-\frac{2C}{n^2}}.(\delta_n^{\omega}(U^c))^{n(n-1)}.\left(\pi.\left(\frac{
C}{2n^4}\right)^2\right)^n.
$$
Therefore by \eqref{eqn:limit} and \eqref{re:thecostant}, we
have
$$
\liminf_{n\to \infty}\frac{1}{n^2}\log\P[\X_n(\sqrt n U)=0]\ge
-\inf_{\mu\in \mathcal {P}(\C \backslash U)}R_{\mu}+\frac{3}{4} = -R_U+\frac{3}{4}.
$$
 The above inequality and \eqref{inequality} give the result.
\end{proof}

It remains to prove  Lemma \ref{lem:feketedistance}.

\begin{proof}[Proof of Lemma \ref{lem:feketedistance}]
Let $P(z)=(z-z_2^*)\cdots (z-z_n^*)$. Now we show that
\begin{eqnarray*}\label{eqn:feketedist}
\min\{|z_1^*-z_k^*|: 2\le k \le n\}\ge C.\frac{1}{n^3}
\end{eqnarray*}
for some constant $C$. Suppose $|z_1^*-z_2^*|\le \frac{1}{n^2}$.
By Cauchy integral formula we have
\begin{eqnarray*}
|P(z_1^*)|&=&|P(z_1^*)-P(z_2^*)|
\\&=&\left|\frac{1}{2\pi i}\int_{|\zeta-z_1^*|=\frac{2}{n^2}}\frac{P(\zeta)}{(\zeta-z_1^*)}d\zeta-
\frac{1}{2\pi
i}\int_{|\zeta-z_1^*|=\frac{2}{n^2}}\frac{P(\zeta)}{(\zeta-z_2^*)}d\zeta\right|
\\&\le&
\frac{1}{2\pi}\int_{|\zeta-z_1^*|=\frac{2}{n^2}}\frac{|P(\zeta)||z_1^*-z_2^*|}{|\zeta-z_1^*||\zeta-z_2^*|}|d\zeta|
\\&\le
&\frac{1}{2\pi}.|P(\zeta^*)|.\frac{n^2}{2}.n^2.|z_1^*-z_2^*|.2\pi.\frac{2}{n^2},\;\;\;\;\mbox{(as $|\zeta-z_2^*|\ge \frac{1}{n^2}$)}
\end{eqnarray*}
where $\zeta^*\in \{\zeta:|\zeta-z_1^*|=\frac{2}{n^2}\}$ such that
$P(\zeta^*)=\sup\{|P(\zeta)|:|z_1^*-\zeta|=\frac{2}{n^2}\}$. Therefore we have
\begin{eqnarray}\label{eqn:z1}
|P(z_1^*)|\le n^2.|z_1^*-z_2^*|.|P(\zeta^*)|.
\end{eqnarray}
Again we have that if $z,w\in \overline{\D}$ and $|z-w|\le \frac{2}{n}$, then
\begin{eqnarray}\label{eqn:exp}
e^{-(n-1)\frac{|z|^2}{2}}\le 10.e^{-(n-1)\frac{|w|^2}{2}}.
\end{eqnarray}
 Indeed, we have
\begin{eqnarray*}
e^{-\frac{(n-1)}{2}(|z|^2-|w|^2)}= e^{-\frac{(n-1)}{2}(|z|+|w|)(|z|-|w|)}\le e^{\frac{(n-1)}{2}.2.\frac{2}{n}}= e^{2}.
\end{eqnarray*}

\noindent{\bf Case I:} Suppose $\zeta^*\in U^c$. Since  $z_1^*,z_2^*,\ldots,z_n^*$ are the Fekete points for $U^c$ we have 
$$
|P(\zeta^*)|.e^{-(n-1)\frac{|\zeta^*|^2}{2}}\le |P(z_1^*)|.e^{-(n-1)\frac{|z_1^*|^2}{2}}.
$$
Then from \eqref{eqn:z1} and \eqref{eqn:exp} we get
\begin{eqnarray*}
|P(z_1^*)|e^{-(n-1)\frac{|z_1^*|^2}{2}}&\le & n^2.
|z_1^*-z_2^*|.|P(\zeta^*)|.10.e^{-(n-1)\frac{|\zeta^*|^2}{2}}
\\ &\le &10.n^2. |z_1^*-z_2^*|.|P(z_1^*)|.e^{-(n-1)\frac{|z_1^*|^2}{2}}.
\end{eqnarray*}
 And hence we get
\begin{eqnarray*}
|z_1^*-z_2^*|\ge \frac{1}{10.n^2}.
\end{eqnarray*}

\noindent{\bf Case II:} Suppose $\zeta^*\in U$. Therefore dist$(z_1^*,\partial U)=\inf \{|z-z_1^*|\; :\;z\in \partial U\}<\frac{2}{n^2}$. Choose large $n$ such that $\frac{1}{n}<\epsilon $. From the given
condition \eqref{eqn:condition} on $U$, we can choose $\eta\in U^c$ such that  $z_1^*\in\overline{ B(\eta,\frac{1}{n})} \subseteq U^c $. By taking the power series expansion of $P$ around $\eta$ and by triangle inequality, we
get
\begin{eqnarray}\label{eqn:zeta1}
\;\;\;\;|P(\zeta^*)|\le
|P(\eta)|+|\zeta^*-\eta|.\left|\frac{P^{(1)}(\eta)}{1!}\right|+\cdots+|\zeta^*-\eta|^{(n-1)}.\left|\frac{P^{(n-1)}(\eta)}{(n-1)!}\right|,
\end{eqnarray}
where $P^{(r)}(\cdot)$ denotes the $r$-th derivative of $P$. From
the Cauchy integral formula we have
$$
\left|\frac{P^{(r)}(\eta)}{r!}\right|\le
\frac{1}{2\pi}\int_{|z-\eta|=\frac{1}{n}}\frac{|P(z)|}{|z-\eta|^{r+1}}|dz|\le
|P(\eta^*)|.n^r,
$$
where $\eta^*\in \{z:|z-\eta|=\frac{1}{n}\}$ such that
$P(\eta^*)=\sup\{|P(z)|:|z-\eta|=\frac{1}{n}\}$. Note that
$|\zeta^*-\eta|\le |\zeta^*-z_1^*|+|z_1^*-\eta|\le
\frac{2}{n^2}+\frac{1}{n}$, therefore we have
$$
|\zeta^*-\eta|^{r}.\frac{|P^{(r)}(\eta)|}{r!}\le \left(1+\frac{2}{ n}\right)^r.|P(\eta^*)|\le e^{2}.|P(\eta^*)|.
$$
Therefore from \eqref{eqn:zeta1} we get
$$
|P(\zeta^*)|\le |P(\eta)|+e^{2}.n.|P(\eta^*)|.
$$
And hence from \eqref{eqn:z1} and \eqref{eqn:exp} we have
\begin{eqnarray*}
&&|P(z_1^*)|e^{-(n-1)\frac{|z_1^*|^2}{2}}
\\&\le& n^2.|z_1^*-z_2^*|.10.\left(|P(\eta)|e^{-(n-1)\frac{|\eta|^2}{2}}+n.e^{2}.|P(\eta^*)|e^{-(n-1)\frac{|\eta^*|^2}{2}}\right)
\\&\le& n^2.|z_1^*-z_2^*|.10.\left(1+n.e^{2}\right)|P(z_1^*)|e^{-(n-1)\frac{|z_1^*|^2}{2}},
\end{eqnarray*}
since $z_1^*,z_2^*,\ldots,z_n^*$ are the Fekete points for $U^c$
and $\eta,\eta^*\in U^c$. Therefore we get
\begin{eqnarray*}
|z_1^*-z_2^*|\ge \frac{1}{2.10.e^{2}n^3}.
\end{eqnarray*}
By Case I and Case II we get that if $|z_1^*-z_2^*|\le
\frac{1}{n^2}$, then $|z_1^*-z_2^*|\ge \frac{20.e^{2}}{n^3}$. Similarly,
if $|z_1^*-z_k^*|\le \frac{1}{n^2}$ for $k=2,3,\ldots,n$, then
$|z_1^*-z_k^*|\ge \frac{20.e^{2}}{n^3}$. Therefore we have
$$
\min\{|z_1^*-z_k^*|: k=2,3,\ldots,n\}\ge \frac{20.e^{2}}{n^3}.
$$
Similarly it can be shown that $|z_{\ell}^*-z_k^*|\ge
\frac{20.e^{2}}{n^3}$ for all $1\le \ell\neq k \le n$ and hence
$$
\min\{|z_{\ell}^*-z_k^*|:1\le  \ell\neq k \le n \}\ge
\frac{20.e^{2}}{n^3}.
$$
Hence the result.
\end{proof}

Now we proceed to prove Lemma \ref{thm:holeprobability}, which helps in the proof of Theorem \ref{thm:holeprobability12}.

\begin{proof}[Proof of Lemma \ref{thm:holeprobability}]

As we already have the  upper bound from \eqref{inequality}, it only remains to prove the lower bound. From \eqref{eqn:holeprob}, we have
\begin{eqnarray}\label{eqn:lower}
\P[\X_n(\sqrt n U)=0]&=&\frac{1}{Z_n}\int_{U^c}\ldots \int_{U^c}
e^{-n\sum_{k=1}^n|z_k|^2}\prod_{i<j}|z_i-z_j|^2\prod_{k=1}^n
dm(z_k)\nonumber
\\&\ge &\frac{1}{Z_n}\int_{U^c}\ldots \int_{U^c}
e^{-n\sum_{k=1}^n|z_k|^2}\prod_{i<j}|z_i-z_j|^2\prod_{k=1}^n\frac{f(z_k)}{M}
dm(z_k),\nonumber
\end{eqnarray}
where $f$ is a compactly supported  probability density function with support in  $U^c$ and bounded by $M$. Applying logarithm on both sides we have
\begin{eqnarray*}
&&\log \P[\X_n(\sqrt n U)=0]
\\&\ge& -\log (Z_n .M^n) + \log \left(\int_{U^c}\ldots \int_{U^c}
e^{-n\sum_{k=1}^n|z_k|^2}\prod_{i<j}|z_i-z_j|^2\prod_{k=1}^n f(z_k)
dm(z_k)\right)
\\&\ge& -\log (Z_n .M^n) +  \int_{U^c}\ldots \int_{U^c}\log\left(
e^{-n\sum_{k=1}^n|z_k|^2}\prod_{i<j}|z_i-z_j|^2\right)\prod_{k=1}^n f(z_k)
dm(z_k)
\\&=& -\log (Z_n .M^n)+n(n-1)\int_{U^c} \int_{U^c} (\log |z_1-z_2|-\frac{n}{n-1}|z_1|^2)\prod_{k=1}^2 f(z_k)dm(z_k),
\end{eqnarray*}
where the second inequality follows from Jensen's inequality.
Therefore by taking limits on both sides, we have
\begin{eqnarray}\label{eqn:bddsupport}
\liminf_{n\to \infty}\frac{1}{n^2}\log \P[\X_n(\sqrt n U)=0]&\ge&
-\lim_{n\to \infty}\frac{1}{n^2}\log Z_n - R_{\mu}= - R_{\mu}+\frac{3}{4}
\end{eqnarray}
for any  probability measure $\mu$ with density bounded and
compactly supported on $U^c$.


Let $\mu$ be probability measure with density $f$ compactly
supported on $U^c$.  Consider the sequence of measures with
bounded densities $$d{\mu}_M(z)= \frac{f_M(z) dm(z)}{\int
f_M(w)dm(w) },$$ where $f_M(z)=\min\{f(z),M\}$. From monotone convergence theorem for positive and negative parts of logarithm, it follows that
$$\lim_{M \to \infty} \int_{U^c} \int_{U^c} \log|z_1-z_2|\prod_{k=1}^2 f_M(z_k)dm(z_k) = \int_{U^c} \int_{U^c} \log|z_1-z_2|\prod_{k=1}^2 f(z_k)dm(z_k).$$ From monotone convergence theorem, it follows that $\lim_{M \to \infty} \int
f_M(w)dm(w)=1$ and $\lim_{M \to \infty} \int
|z|^2f_M(z)dm(w)=\int
|z|^2f(z)dm(w)$. Therefore 
$$\lim_{M \to \infty} R_{\mu_{M}} = R_{\mu}.$$ So \eqref{eqn:bddsupport} is true for any measure with density compactly supported on $U^c$.

Let $\mu$ be a probability measure with compact support at a
distance of at least $\delta$ from $U$. Then the convolution
$\mu*\sigma_{\epsilon} $,  where $\sigma_{\epsilon}$ is  uniform
probability measure  on  disk of radius $\epsilon$ around origin,
has density compactly supported in $U^c$, if $\epsilon$ is less
than $\delta$. We have
\begin{eqnarray*}
I_{\mu*\sigma_{\epsilon}}&=&\iint \log|z-w|d(\mu*\sigma_{\epsilon})(z)d(\mu*\sigma_{\epsilon})(w)
\\&=&\iint\iint\iint \log|z+\epsilon r_1 e^{i \theta_1}-w-\epsilon r_2 e^{i \theta_2}|\frac{r_1dr_1d\theta_1}{\pi}\frac{r_2dr_2d\theta_2}{\pi}d\mu(z)d\mu(w)
\\&&\mbox{(limits of $r_1,r_2$ are from $0$ to 1 and $\theta_1,\theta_2$ are from $0$ to $2\pi$ )}
\\&\ge &\iint \log |z-w| d\mu(z)d\mu(w),
\end{eqnarray*}
where the inequality follows from the repeated application of the mean value property of the subharmonic function $\log |z|$. And also we have
\begin{eqnarray*}
I_{\mu*\sigma_{\epsilon}}\le \iint \log [|z-w| +2\epsilon]d\mu(z)d\mu(w).
\end{eqnarray*}
Therefore, $\lim_{\epsilon\to 0}I_{\mu*\sigma_{\epsilon}}=I_{\mu} $ and hence $\lim_{\epsilon\to 0}R_{\mu*\sigma_{\epsilon}}=R_{\mu}.$

So, \eqref{eqn:bddsupport} is true for any probability measure with compact support
whose distance from $U$ is positive. Hence the required lower bound.
\end{proof}

It remains to prove the Theorem \ref{thm:holeprobability12}.
\begin{proof}[Proof of Theorem \ref{thm:holeprobability12}]
Let $U,U_1,U_2,\ldots $ be open subsets of $\D$ satisfying conditions
in Theorem \ref{thm:holeprobability12}. By Lemma
\ref{thm:holeprobability}, we have
\begin{eqnarray*}
\limsup_{n\to \infty}\frac{1}{n^2}\log\P[\X_n(\sqrt n U)=0]&\le
&-\inf_{\mu\in {\mu\in
\mathcal P (U^c)}}R_{\mu}+\frac{3}{4},
\\ \liminf_{n\to
\infty}\frac{1}{n^2}\log\P[\X_n(\sqrt n U)=0]&\ge &-\inf_{\mu\in
\mathcal A}R_{\mu}+\frac{3}{4} \ge -\inf_{\mu\in \mathcal
A_m}R_{\mu}+\frac{3}{4},
\end{eqnarray*}
where
$\mathcal A=\{\mu\in \mathcal P(\mathbb C): \mbox {dist(Supp($\mu$),$\overline
U$)$>0$\}},$ $\mathcal A_m=\{\mu\in \mathcal P(\mathbb C):\mu(U_m)=0\}$. Since $\overline{U} \subset U_m$, we have
  $\mathcal A_m\subset \mathcal A $. By Theorem
\ref{thm:generalformula}, we also have that
$$
R_{U_m}=\inf_{\mu\in \mathcal A_m}
R_{\mu}=\frac{3}{4}+\frac{1}{2}\left[\int_{\partial
U_m}|z|^2d\nu_m(z)-\frac{1}{\pi}\int_{{U_m}}|z|^2dm(z)\right].
$$
Since the  balayage measures $\nu_m$ converge weakly to the balayage
measure $\nu$, $\inf_{\mu\in \mathcal A_m}R_{\mu}$ converges to
$\inf_{\mu\in \mathcal A}R_{\mu}$ as $m\to\infty$. Therefore we
have
\begin{eqnarray*}
-R_{U}+\frac{3}{4}&\ge&\limsup_{n\to
\infty}\frac{1}{n^2}\log\P[\X_n(\sqrt n U)=0]
\\&\ge& \liminf_{n\to
\infty}\frac{1}{n^2}\log\P[\X_n(\sqrt n U)=0]
\ge-R_U+\frac{3}{4}.
\end{eqnarray*}
Hence the result.
\end{proof}

\begin{remark}
Convex open sets in unit disk, which do not intersect unit circle, satisfy the conditions of Theorem
\ref{thm:holeprobability12}. More generally  note that if $U$ is
an open set containing origin such that $\overline U \subset aU$ for
all $a>1$, then the balayage measure $\nu_a$ on $\partial (aU)$ is
given in terms of the balayage measure $\nu$ on $\partial U$ as
 $\nu_a (B)=\frac{1}{a^2}\nu(\frac{1}{a}B)$ for any measurable set $B\subset \C$. Therefore $\nu_a$
 converges weakly to $\nu$ as $a\to 1$. Similarly, translations of such open sets $U$ also
 satisfy the conditions  of Theorem
\ref{thm:holeprobability12}. All the examples we
 considered, except annulus, satisfy the above condition. In the
 case of annulus, we can construct the required sequence of open
 sets by varying inner and outer radii of annulus appropriately.
\end{remark}

In the  next section we  calculate potential,
 the balayage measure $\nu_2$ and the constant $R_U$ explicitly for some particular open sets $U$.

\section{Examples}\label{example}
In this section we calculate the balayage measure $\nu_2$ and the
constant $R_U$ explicitly for some particular open sets $U$. In the first example we consider annulus with inner and outer radius $a$ and $b$ respectively.

\begin{example}\label{thm:annulus}
Fix $0<a<b<1$. Suppose $U=\{z\in \mathbb C: a<|z|<b\}$.
Then the balayage measure is $\nu_2=\nu_2'+\nu_2''$, where
\begin{eqnarray*}
 d\nu_2'(z)&=&\left\{\begin{array}{lr}
\lambda(b^2-a^2)\frac{d\theta}{2\pi}& \mbox{ if  } |z|=a\\0&
\mbox{o.w.}
\end{array}\right.
\\d\nu_2''(z)&=&\left\{\begin{array}{lr}
(1-\lambda)(b^2-a^2)\frac{d\theta}{2\pi}& \mbox{ if  } |z|=b\\0&
\mbox{o.w.}
\end{array}\right.
\end{eqnarray*}
and  $\lambda $ is given by
$$
\lambda = \frac{(b^2-a^2)-2a^2\log(b/a)}{2(b^2-a^2)\log(b/a)}.
$$
The constant is
$$
R_{U}=\frac{3}{4}+\frac{1}{4}(b^4-a^4)-\frac{1}{4}\frac{(b^2-a^2)^2}{\log(b/a)}.
$$
\end{example}

\noindent {\bf Note: }In particular if $a=b$, then
$R_{U}=\frac{3}{4}$ which implies Fact \ref{thm:disk}. Again if
we take aspect ratio $a/b=c$, then
$$
R_{U}=\frac{3}{4}+\frac{1}{4}\left((1-c^4)+\frac{(1-c^2)^2}{\log
c}\right)b^4.
$$
Note that the same  expression has appeared  in hole probability for infinite Ginibre ensemble (Theorem \ref{thm:annulus1}).

\begin{proof}[Computation for Example \ref{thm:annulus}] Because of rotational symmetry, the balayage measure on $\partial U$ has to be of the form 
$\nu_2'+\nu_2''$, where   
\begin{eqnarray*}
 d\nu_2'(z)&=&\left\{\begin{array}{lr}
\lambda(b^2-a^2)\frac{d\theta}{2\pi}& \mbox{ if  } |z|=a\\0&
\mbox{o.w.}
\end{array}\right.
\\d\nu_2''(z)&=&\left\{\begin{array}{lr}
(1-\lambda)(b^2-a^2)\frac{d\theta}{2\pi}& \mbox{ if  } |z|=b\\0&
\mbox{o.w.}
\end{array}\right.
\end{eqnarray*} 
let $\mu_2$ be the measure $\frac{1}{\pi}m$ on $U$. Note that if $|z|>b$, then $p_{\nu_2}(z)=p_{\mu_2}(z)$ for every choice of $0<\lambda<1$.
If $|z|<a$, by Fact \ref{ft:fundamental}, we have 
\begin{eqnarray*}
p_{\nu_2}(z)&=&\lambda(b^2-a^2)\log \frac{1}{a}+(1-\lambda)(b^2-a^2)\log \frac{1}{b},
\\p_{\mu_2}(z)&=&b^2\log\frac{1}{b}-a^2\log\frac{1}{a}+\frac{1}{2}(b^2-a^2).
\end{eqnarray*}
 By equating $p_{\nu_2}(z)=p_{\mu_2}(z)$ when  $|z|<a$, we get 
$$
\lambda=\frac{(b^2-a^2)-2a^2\log(b/a)}{2(b^2-a^2)\log(b/a)}.
$$
 Therefore for this particular choice of $\lambda$ we have $p_{\nu_2}(z)=p_{\mu_2}(z)$ for all $z \in \{z: a\le |z|\le b\}^c$. Therefore this particular value of $\lambda$ gives us the equilibrium measure and the constant
$$
R_{U}=\frac{3}{4}+\frac{1}{4}(b^4-a^4)-\frac{1}{4}\frac{(b^2-a^2)^2}{\log(b/a)},
$$
can be verified from the formula given in the Theorem \ref{thm:generalformula}.
\end{proof}

In the next example we consider disk of radius $a$ contained in unit disk. Let $B(c_0,a)$ be the ball of radius $a$ centered at $c_0$.

 \begin{example}\label{ex=cB}
For $U=B(c_0,a)\subseteq \D$,
 the equilibrium measure is $\nu=\nu_1+\nu_2$, where
 $$
d\nu_1(z)=\left\{\begin{array}{lr}
\frac{1}{\pi}dm(z)& \mbox{  if  } z\in\overline{\mathbb D}\backslash U\\0& \mbox{o.w.}
\end{array}\right.
\hspace{.5cm} \mbox{ and } \;\;
d\nu_2(z)=\left\{\begin{array}{lr}
\frac{a^2}{2\pi}d\theta& \mbox{  if  } |z-c_0|=a\\0& \mbox{o.w.}
\end{array}\right.
$$
and the constant is $R_U=\frac{3}{4}+\frac{1}{4}a^4$.

\end{example}

Note that the equilibrium measure and the constant do not depend
on the position of the ball. These values depend only on radius of
the ball. This follows directly from the fact that the balayage measure corresponding to uniform measure on a ball is uniform on its boundary, which follows easily from Fact \ref{ft:fundamental}. 
 Now we consider  ellipse.

\begin{example}\label{ex:ellipse}
Fix $0<a,b<1$. Suppose
$U=\{(x,y) | \frac{x^2}{a^2}+\frac{y^2}{b^2}< 1\}$. Then
$$
d\nu_2(z)=\frac{ab}{2\pi}\left[1-{\frac{a^2-b^2}{a^2+b^2}}\cos(2\theta)\right]d\theta\;\;\mbox{and
}\;\; R_U=\frac{3}{4}+\frac{1}{2}\cdot\frac{(ab)^3}{a^2+b^2}.
$$
\end{example}

\begin{proof}[Computation for Example \ref{ex:ellipse}]
Let $x=ar\cos{\theta},  y=br\sin{\theta}, 0 \leq \theta \leq 2\pi, 0
\leq r \leq 1$. Then we have
\begin{eqnarray*}
\frac{1}{\pi}\int_{U}w^ndm(w)&=&\frac{1}{\pi}\int_{U}(x+iy)^n dx dy\\
  &=&\frac{1}{\pi} \int_{0}^{1}\int_{0}^{2\pi}r^n(a\cos{\theta}+ib\sin{\theta})^n abr d{\theta}dr
\\&=&\frac{1}{\pi}\int_{0}^{2\pi}\frac{ab(a\cos{\theta}+ib\sin{\theta})^n}{n+2}d{\theta}
\\&&(\mbox{by substituting $\alpha=\frac{a+b}{2}, \beta=\frac{a-b}{2}$})
\\&=&\frac{1}{\pi}\int_{0}^{2\pi}\frac{({\alpha}^2-{\beta}^2)(\alpha e^{i\theta}+\beta e^{-i\theta})^n}{n+2}d{\theta} \\
\\&=& \left\{\begin{array}{lr}
\frac{1}{\pi}\cdot\frac{({\alpha}^2-{\beta}^2){\alpha}^{n/2}{\beta}^{n/2}{n \choose
n/2}}{n+2} & \mbox{if $n$ is even}\vspace{.3cm}
\\0 & \mbox{if $n$ is odd.}
\end{array}\right.
\end{eqnarray*}

Let $d\nu_2(w)=\frac{1}{\pi}[c_0+c_1(e^{2i\theta}+e^{-2i\theta})]d\theta$, then we have
\begin{eqnarray*}
\int_{\partial U}w^nd\nu_2(w)&=&\frac{1}{\pi}\int_{0}^{2\pi}{(\alpha e^{i\theta}+\beta e^{-i\theta})^n}[c_0+c_1(e^{2i\theta}+e^{-2i\theta})]d{\theta}\\
&=& \left\{\begin{array}{lr}\frac{1}{\pi}\left[ c_0{\alpha}^{n/2}{\beta}^{n/2}{n \choose
n/2}+c_1{n \choose
n/2-1}{\alpha}^{n/2}{\beta}^{n/2}(\frac{\alpha^2+\beta^2}{\alpha\beta})\right]
& \mbox{if $n$ is even}\vspace{.3cm}
\\0 & \mbox{if $n$ is odd.}
\end{array}\right.\\
\end{eqnarray*}
Note that if we take
$$({\alpha}^2-{\beta}^2)\frac{{n\choose n/2}}{n+2}= c_0{n \choose
n/2}+c_1{n \choose
n/2-1}\left(\frac{\alpha^2+\beta^2}{\alpha\beta}\right),\; \mbox{ for all $n$
even}
$$
which implies that
$$
\frac{({\alpha}^2-{\beta}^2)}{2}=c_0=-c_1\left(\frac{\alpha^2+\beta^2}{\alpha\beta}\right).
$$
Therefore $\nu_2$ satisfies \eqref{eqn:moment} for all $n$ and also has continuous density with respect to arclength of $\partial U$. Therefore, by Remark \ref{relations},  the   measure $\nu_2$ on $\partial U$  given by
$$
d\nu_2(z)=\frac{ab}{2\pi}\left[1-{\frac{a^2-b^2}{a^2+b^2}}\cos(2\theta)\right]d\theta
$$
is the balayage measure on $\partial U$ and constant $R_U$ is given by
$$
R_U=\frac{3}{4}+\frac{1}{2}\left[\int_{\partial
U}|w|^2d\nu_2(w)-\frac{1}{\pi}\int_{ U}|w|^2dm(w)\right]=\frac{3}{4}+\frac{1}{2}\cdot\frac{(ab)^3}{a^2+b^2}.
$$
Hence the result.
\end{proof}

Note that if we take $a=b$ then we get the Example \ref{ex=cB}. In
the next example we consider cardioid.

\begin{example}\label{ex:cardiod}
Fix $a,b>0$ such that 
$U=\{re^{i\theta}|0\leq r< b(1+2a\cos\theta), 0\le \theta\le 2\pi\} \subseteq \D$. Then the balayage measure $\nu_2$ and the constant $R_U$ are given by
$$
d\nu_2(w)=\frac{b^2}{\pi}(1+a^2+2a\cos \theta)d\theta\;\mbox{ and } R_U=\frac{3}{4}+\frac{b^4}{2}\left((a^2+1)^2-\frac{1}{2}\right).
$$
\end{example}

\noindent {\bf Note: }The cardioid $U$ can be thought of as small
perturbation of disk of radius $b$.

\begin{proof}[Computation for  Example \ref{ex:cardiod}]
We have
\begin{eqnarray*}
\frac{1}{\pi}\int_{U}w^ndm(w)&=&\frac{1}{\pi}\int_{0}^{2\pi}\int_{0}^{b(1+2a\cos{\theta})}r^ne^{in\theta}rdrd\theta\\
&=&\frac{1}{\pi}\int_{0}^{2\pi}\frac{b^{n+2}(1+2a\cos{\theta})^{n+2}}{n+2}e^{in\theta}d\theta\\
&=&\frac{b^{n+2}}{\pi}\int_{0}^{2\pi}\frac{(1+ae^{i\theta}+ae^{-i\theta})^{n+2}}{n+2}e^{in\theta}d\theta\\
&=&\frac{b^{n+2}}{\pi(n+2)}\int_{0}^{2\pi}\sum_{0\leq{u+v}\leq{n+2}}\frac{(n+2)!}{u!v!(n+2-u-v)!}a^{u+v}e^{i(n+u-v)\theta}d\theta\\
&=&\frac{b^{n+2}}{\pi(n+2)}\left(\frac{(n+2)!}{0!n!2!}a^{n}+\frac{(n+2)!}{1!(n+1)!0!}a^{n+2}\right).2{\pi}\\
&=&{b^{n+2}}\cdot a^n(n+1+2a^2).
\end{eqnarray*}
We show that the measure
$d\mu(w)=\frac{b^2}{2\pi}(1+2a^2+ae^{i\theta}+ae^{-i\theta})d\theta$ satisfies the
required condition \eqref{eqn:moment} to be the  balayage measure for cardioid. We have
\begin{eqnarray*}
\int_{\partial U}w^nd\mu(w)&=&\frac{b^{n+2}}{2\pi}\int_{0}^{2\pi}{(1+ae^{i\theta}+ae^{-i\theta})^n}e^{in\theta}(1+2a^2+ae^{i\theta}+ae^{-i\theta})d\theta
\\&=&\frac{b^{n+2}}{2\pi}\left(a^n(1+2a^2)+a\frac{n!}{1!0!(n-1)!}a^{n-1}\right).{2\pi}\\
&=&{b^{n+2}}.a^n(1+2a^2+n)
\\&=&\frac{1}{\pi}\int_{U}w^ndm(w) \;\;\mbox{ for all $n$. }
\end{eqnarray*}
 Therefore $\mu$ satisfies \eqref{eqn:moment} for all $n$ and also has continuous density with respect to arclength of $\partial U$. Therefore, by Remark \ref{relations},  the balayage  measure $\nu_2$ on boundary of $U$ is given by
 $$
d\nu_2(w)=\frac{b^2}{2\pi}(1+2a^2+ae^{i\theta}+ae^{-i\theta})d\theta.
$$
Therefore we have
\begin{eqnarray*}
&&\int_{\partial U}|w|^2d\nu_{2}(w)-\frac{1}{\pi}\int_{U}|w|^2dm(w)
\\&=&\frac{b^4}{2\pi}\int_{0}^{2\pi}{(1+2a\cos{\theta})^2}(1+2a^2+2a\cos{\theta})d\theta-\frac{1}{\pi}\int_{0}^{\pi}\int_{0}^{b(1+2a\cos{\theta})}r^3drd\theta\\
&=&b^4\left((a^2+1)^2-\frac{1}{2}\right).
\end{eqnarray*}
Hence we get the required constant $R_U$ from \eqref{eqn:robinconstant}.
\end{proof}

 In the next few examples, we could not find the balayage measure explicitly, however we calculated the constant $R_U$ explicitly.

\begin{example}\label{ex:triangle}
Fix $0<a<1$. Suppose $U=aT$, where $T$ be triangle with cube roots of unity $1,\omega,\omega^2$ as
vertices. Then the constant is
$$
R_U=\frac{3}{4}+\frac{a^4}{2\pi}\cdot \frac{9\sqrt{3}}{80}.
$$
\end{example}

\begin{proof}[Computation for Example \ref{ex:triangle}]
The region $T$ can be written as
$$
T=\{r(tw^p+(1-t)w^{p+1})|0\leq r < 1, 0\leq t \leq 1, p=0,1,2\}.
$$
Suppose $x+iy=ar(t\omega^p+(1-t)\omega^{p+1})$. Then by change of variables, we have
$$
\frac{1}{\pi}dm(z)=\frac{1}{\pi} dxdy=\frac{1}{\pi}\frac{\sqrt{3}}{2}a^2rdrdt.
$$
Let $d\nu_2(t)$ be the balayage measure on the boundary of triangle $T$. Then from \eqref{eqn:moment}, we get
$$
\int_{\partial U}z^nd\nu_2(z)=\frac{1}{\pi}\int_{ U}z^ndm(z), \;\;\mbox{for all }\; n\geqslant 0.
$$
Which implies for all $n\ge 0$,
\begin{eqnarray*}
&&\int_{0}^{1}(t+(1-t)\omega)^n(1+\omega^n+\omega^{2n})d\nu_2(t)
\\&=&\int_{0}^{1}\int_{0}^{1}r^n(t+(1-t)\omega)^n(1+\omega^n+\omega^{2n})\frac{\sqrt{3}}{2\pi}a^2rdrdt.
\end{eqnarray*}
Since $1+\omega^n+\omega^{2n}=0$ when $n$ is not multiple of $3$. Therefore we get
\begin{eqnarray}\label{eqn:triangle}
\int_{0}^{1}(t+(1-t)\omega)^{3n}d\nu_2(t)= \int_{0}^{1}
(t+(1-t)\omega)^{3n}\frac{\sqrt{3}a^2}{2\pi(3n+2)}dt.
\end{eqnarray}
for all $n\ge 0$. This is the key equation to calculate the balayage measure on $\partial U$.
Solve this equation, we can get the balayage measure on $\partial U$. But we could not solve this equation.

We manage to calculate the constant $R_U$ using \eqref{eqn:triangle}.
By putting $n=1$ and comparing the real parts in both side of \eqref{eqn:triangle}, we have
\begin{eqnarray*}
\int_{0}^1(1-\frac{9}{2}t(1-t))d\nu_2(t)=\int_{0}^1(1-\frac{9}{2}t(1-t))\frac{\sqrt{3}a^2}{10\pi}dt.
\end{eqnarray*}
(As real part of $(t+(1-t)\omega)^{3}$ is $(1-\frac{9}{2}t(1-t))$. By using the fact that $\int_0^1 d\nu_2(t)=\frac{\sqrt{3}a^2}{4\pi}$
and simplifying the last equation we get
\begin{eqnarray}\label{eqn:nu2}
\int_0^1 t(1-t))d\nu_2(t)=\frac{\sqrt{3}a^2}{20\pi}.
\end{eqnarray}

\noindent Therefore we have
\begin{eqnarray*}
&&\int_{\partial U}|z|^2d\nu_2(z)-\frac{1}{\pi}\int_U |z|^2 dm(z)
\\&=&3\left[\int_{0}^1|ar(t+(1-t)\omega)|^2d\nu_2(t)-\int_0^1\int_0^1|ar(t+(1-t)\omega)|^2\frac{\sqrt{3}a^2}{2\pi}rdrdt\right]
\\&=&3a^2\left[\int_{0}^1[1-3t(1-t)]d\nu_2(t)-\frac{\sqrt{3}a^2}{2\pi}\int_0^1\int_0^1[1-3t(1-t)]r^3drdt\right]
\\&=&3a^2\left(\frac{\sqrt{3}a^2}{10\pi}-\frac{\sqrt{3}a^2}{16\pi}\right)\;\hspace{2cm}\mbox{( by \eqref{eqn:nu2} )}
\\&=&\frac{9\sqrt{3}a^4}{80\pi}.
\end{eqnarray*}
Hence we have the required constant $R_U$ from \eqref{eqn:robinconstant}.
\end{proof}

In the next example, we consider semi-disk. We only calculated the constant $R_U$. We were unable to find equilibrium measure.

\begin{example}\label{ex:halfdisk}
Fix $0<a<1$. Suppose $U=\{re^{i\theta}:0<r<a,0<\theta<\pi\}$. Then the constant is
$$
R_U=\frac{3}{4}+\frac{a^4}{2}\left(\frac{1}{2}-\frac{4}{\pi^2}\right).
$$
\end{example}

\begin{proof}[Computation for Example \ref{ex:halfdisk}]
Let $\nu_2=\nu_2'+\nu_2''$ be the balayage measure. Where $\nu_2'$  is the measure on diameter of semicircle ( $\{re^{i\theta }: r\le a, \theta=0,\pi\}$) and $d\nu_2''(z)=g(\theta)d\theta$
is the  measure on circular arc ($\{ae^{i\theta}:0\le \theta \le \pi\}$). Then from \eqref{eqn:moment}, we get
$$
\int_{\partial U}z^nd\nu_2(z)=\frac{1}{\pi}\int_{ U}z^nm(z), \;\;\mbox{for all }\; n\geqslant 0.
$$
Which implies for all $n\ge 0$,
\begin{eqnarray*}
&&\int_{-a}^{a}t^nd\nu_2'(t)+\int_{0}^{\pi}a^ne^{in\theta}g(\theta)d\theta
=\int_{0}^{\pi}\int_{0}^{a}r^ne^{in\theta}rdr\frac{d\theta}{\pi} .
\end{eqnarray*}
 Therefore we get
\begin{eqnarray}\label{eqn:semicircle}
\int_{-a}^{a}t^nd\nu_2'(t)+\int_{0}^{\pi}a^ne^{in\theta}g(\theta)d\theta=
 \left\{\begin{array}{lr}\frac{2ia^{n+2}}{n(n+2)\pi}&\mbox{$n$ is odd}\\o&\mbox{$n$ is even}\end{array}\right.
\end{eqnarray}
 This is the key equation to calculate the balayage measure   on $\partial U$. In principle if we solve this equation then we get the balayage measure on $\partial U$.
 But we could not solve this equation.

However, with out calculating balyage measure we are  able to calculate
the constant $R_U$ using \eqref{eqn:semicircle}. Comparing
imaginary part in both side of \eqref{eqn:semicircle}, we have
\begin{eqnarray}\label{eqn:relation}
\int_{0}^{\pi}\sin{n\theta}g(\theta)d\theta= \left\{\begin{array}{lr}\frac{2a^{2}}{n(n+2)\pi}&\mbox{$n$ is odd}\\o&\mbox{$n$ is even}\end{array}\right.
\end{eqnarray}
Since $g(\theta)$ is defined on $[0,\pi]$, its fourier series can be made to contain only sine terms, and moreover because of its
symmetry with respect to $\frac{\pi}{2}$, $g(\theta)$'s fourier series contains only odd sine terms. Therefore we get
\begin{eqnarray}\label{eqn:series}
g(\theta)=\frac{4a^2}{\pi^2}\sum_{k=1}^{\infty} \frac{\sin{(2k-1)\theta}}{4k^2-1}.
\end{eqnarray}

\noindent Using \eqref{eqn:semicircle}, \eqref{eqn:relation} and \eqref{eqn:series} we get
\begin{eqnarray*}
\int_{\partial U}|z|^2d\nu_2(z)&=& \int_{-a}^{a}t^2d\nu_2'(t)+\int_{0}^{\pi}a^2g(\theta)d\theta
\\&=&a^2\int_{0}^{\pi}(1-\cos{2\theta})g(\theta)d\theta
\\&=&\frac{4a^4}{{\pi}^2}\left[\sum_{k=1}^{\infty}\frac{1}{(4k^2-1)}\left(\frac{2}{2k-1}-\frac{1}{2k+1}-\frac{1}{2k-3}\right)\right]
\\&=&\frac{2a^4}{{\pi}^2}\left(\frac{3\pi^2}{8}-2\right)
\\&=&a^4\left(\frac{3}{4}-\frac{4}{{\pi}^2}\right).
\end{eqnarray*}
Therefore  the constant
$$
R_U=\frac{3}{4}+\frac{1}{2}\left[\int_{\partial U}|z|^2d\nu_2(z)-\frac{1}{\pi}\int_{ U}|z|^2dm(z)\right]=\frac{3}{4}+\frac{a^4}{2}\left(\frac{1}{2}-\frac{4}{{\pi}^2}\right).
$$
Hence the desired result.
\end{proof}

\section{Proof of Theorem \ref{thm:annulus1}}\label{sec:annulus}

In this section we give the proof of Theorem \ref{thm:annulus1} using the result of Kostlan \cite{kostlan}.   
\begin{result}[Kostlan]\label{lem:kostlan}
The set of absolute values of the eigenvalues of $G_n$ has the same distribution as $\{R_1,R_2,\ldots,R_n\} $ where $R_k$ are independent and $R_k^2\sim \mbox{Gamma}(k,1)$.
\end{result}
As a corollary the set of absolute values of the points of ${\mathcal X}_{\infty}$  has the same distribution as $\{R_1,R_2,\ldots \},$ where $R_k^2\sim \mbox{Gamma}(k,1)$ and all the $R_k$s are independent.

 We state a lemma which will be used in the proof of Theorem \ref{thm:annulus1}.

\begin{lemma}\label{lem:iqn}
Fix $c$ such that $0<c<1$. Then
\begin{eqnarray}\label{eqn:inq1}
&&\P(R_k^{2}<c^2r^2)\le e^{-k\log(\frac{k}{c^2r^2})+k-c^2r^2} \;\mbox{for all } k> c^2r^2.
\\&&\P(R_k^{2}>r^2)\le e^{-r^2+k-k\log(\frac{k}{r^2})}\;\mbox{ for all } k<r^2.\label{eqn:inq2}
\end{eqnarray}
\end{lemma}
Using Result~\ref{lem:kostlan} and Lemma~\ref{lem:iqn} we  prove Theorem \ref{thm:annulus1}. 

\begin{proof}[Proof of Theorem \ref{thm:annulus1}]
{\bf Upper bound: }From the consequence of Result \ref{lem:kostlan} we have
\begin{eqnarray*}
\P[n_c(r)=0]&=&\prod_{k=1}^{\infty}\left(\P(R_k^2<c^2r^2)+\P(R_k^2>r^2)\right)
\\&\le&\prod_{c^2r^2}^{r^2}\left(\P(R_k^2<c^2r^2)+\P(R_k^2>r^2)\right)
\\&\le&\prod_{c^2r^2}^{r^2}\left(e^{-k\log(\frac{k}{c^2r^2})+k-c^2r^2}+e^{-r^2+k-k\log(\frac{k}{r^2})}\right) 
\end{eqnarray*}
Last inequality is follows from Lemma \ref{lem:iqn}. Let $\lambda=-\frac{1-c^2}{2\log c}$. Then
$$
-k\log\left(\frac{k}{c^2r^2}\right)+k-c^2r^2\le -r^2+k-k\log\left(\frac{k}{r^2}\right) \;\; \mbox{ for }k\ge \lambda r^2,
$$
and the opposite inequality holds for $k\le \lambda r^2$.
It is clear that $c^2<\lambda<1 $ for $0<c<1$.
Therefore we have
\begin{eqnarray}\label{eqn:prod}
\P[n_c(r)=0]
&\le&2^{r^2-c^2r^2} \prod_{k=c^2r^2}^{\lambda r^2}e^{-k\log(k/r^2)+2k\log c+k-c^2r^2}\prod_{k=\lambda r^2}^{r^2}e^{-r^2+k-k\log(k/r^2)}\nonumber
\\&=&e^{o(r^4)}\prod_{k=c^2 r^2}^{r^2}e^{-k\log (k/r^2)+k}\prod_{k=c^2 r^2}^{\lambda r^2}e^{2k\log c-c^2r^2}\prod_{k=\lambda r^2}^{r^2}e^{-r^2}.
\end{eqnarray}
Since  $\sum_{k=1}^{r^2}k\log (k/r^2)=r^4\int_{0}^{1}x\log (x) dx+O(r^2\log r)=-\frac{1}{4}r^4+O(r^2\log r)$
 as $r\to \infty$. Hence as $r\to \infty$, we have
\begin{eqnarray}\label{eqn:1}
&&\sum_{c^2r^2}^{r^2}(-k\log(k/r^2)+k)\nonumber
\\&=&\sum_{k=1}^{r^2}(-k\log(k/r^2)+k)-\sum_{1}^{c^2r^2}(-k\log(k/r^2)+k)\nonumber
\\&=&\frac{1}{4}r^4+\frac{1}{2}r^4-\left\{\sum_{1}^{c^2r^2}(-k\log(k/c^2r^2)+k-k\log c^2)\right\}+O(r^2\log r)\nonumber
\\&=&\frac{3}{4}r^4-\left(\frac{1}{4}c^4r^4+\frac{1}{2}c^4r^4-c^4r^4\log c\right)+O(r^2\log r)\nonumber
\\&=&\left(\frac{3}{4}-\frac{3}{4}c^4+c^4\log c\right)r^4(1+o(1)).
\end{eqnarray}
Again as $r\to \infty$, we have
\begin{eqnarray}\label{eqn:2}
\sum_{k=c^2r^2}^{\lambda r^2}(2k\log c -c^2r^2)
&=&(\lambda^2 \log c-c^4\log c+c^4-\lambda c^2)r^4(1+o(1)).
\end{eqnarray}
Therefore by (\ref{eqn:1}) and (\ref{eqn:2}) from (\ref{eqn:prod}) we get
\begin{eqnarray}
&&\prod_{k=c^2r^2}^{\lambda r^2}\P(R_k^2<c^2r^2)\prod_{k=\lambda r^2}^{r^2}\P(R_k^2<r^2)\nonumber
\\&=&\exp\left\{\left(\frac{3}{4}-\frac{3}{4}c^4+c^4\log c+\lambda^2\log c-c^4\log c+c^4-\lambda c^2-1+\lambda \right)r^4(1+o(1))\right\}\nonumber
\\&=&\exp\left\{\left(-\frac{1}{4}(1-c^4)+\lambda(1-c^2)+\lambda^2\log c\right )r^4(1+o(1))\right\}\nonumber
\end{eqnarray}
as $r\to \infty$. Hence we have the following upper bound
\begin{eqnarray}\label{eqn:upperbound}
\P[n_c(r)=0]\le e^{(-\frac{1}{4}(1-c^4)+\lambda(1-c^2)+\lambda^2\log c )r^4(1+o(1))},
\end{eqnarray}
as $r\to \infty$, where $\lambda=-\frac{(1-c^2)}{2\log c}$.

\noindent{\bf Lower bound:} We want to get lower bound for $\P[n_c(r)=0]$ as $r\to \infty$.
%
%
We have
\begin{eqnarray}\label{eqn:initialstep}
&&{}\hspace{.5cm}\P[n_c(r)=0]=\prod_{k=1}^{\infty}\left(\P(R_k^2<c^2r^2)+\P(R_k^2>r^2)\right)\nonumber
\\&&\ge\prod_{k=1}^{c^2r^2}\P(R_k^2<c^2r^2)\prod_{k=c^2r^2}^{\lambda r^2}\P(R_k^2<c^2r^2)\prod_{k=\lambda r^2}^{r^2}\P(R_k^2>r^2)\prod_{k=r^2}^{\infty}\P(R_k^2>r^2).
\end{eqnarray}
We estimate a lower bound for each product term. Since $R_k^2$ has distribution $\mbox{Gamma}(k,1)$ and
$\P(\mbox{Gamma}(k,1)>a)=\P(\mbox{Poisson}(a)<k)$, we have
\begin{eqnarray*}
&&\P(R_k^2>a)=\sum_{j=0}^{k-1}\frac{a^j}{j!}e^{-a}\ge \frac{a^{k-1}}{(k-1)!}e^{-a}
\\&&\P(R_k^2< a)=\sum_{j=k+1}^{\infty}\frac{a^j}{j!}e^{-a}\ge \frac{a^{k+1}}{(k+1)!}e^{-a}.
\end{eqnarray*}
Therefore we have
\begin{eqnarray}\label{eqn:prod1}
&&\prod_{k=c^2r^2}^{\lambda r^2}\P(R_k^{2}<c^2r^2)\prod_{k=\lambda r^2}^{r^2}\P(R_k^{2}>r^2)\nonumber
\\&=&\prod_{k=c^2r^2}^{\lambda r^2}\frac{(c^2r^2)^{k+1}}{(k+1)!}e^{-c^2r^2}\prod_{k=\lambda r^2}^{ r^2}\frac{(r^2)^{k-1}}{(k-1)!}e^{-r^2}\nonumber
\\&=&\exp\left\{(-\frac{1}{4}(1-c^4)+\lambda (1-c^2)+\lambda^2\log c)r^4(1+o(1))\right\}.
\end{eqnarray}
Again as $r\to \infty$, we have
\begin{eqnarray*}
\prod_{k=1}^{r^2}\frac{(r^2)^{k-1}}{(k-1)!}=\exp\left\{\frac{3}{4}r^4+O(r^2\log r)\right\},
\end{eqnarray*}
and consequently we have
\begin{eqnarray}\label{eqn:estimate1}
\prod_{k=c^2r^2}^{r^2}\frac{(r^2)^{k-1}}{(k-1)!}&=&\left(\prod_{k=1}^{r^2}\frac{(r^2)^{k-1}}{(k-1)!}\right).\left(\prod_{k=1}^{c^2r^2}\frac{(r^2)^{k-1}}{(k-1)!}\right)^{-1}\nonumber
\\&=&\left(\prod_{k=1}^{r^2}\frac{(r^2)^{k-1}}{(k-1)!}\right).\left(\prod_{k=1}^{c^2r^2}\frac{(c^2r^2)^{k-1}}{(k-1)!}\right)^{-1}.\left(\prod_{k=1}^{c^2r^2}(c^2)^{k-1}\right)\nonumber
\\&=&e^{\frac{3}{4}r^4+O(r^2\log r)}.e^{-\frac{3}{4}c^4r^4+O(r^2\log r)}e^{c^4r^4\log c+O(r^2)}\nonumber
\\&=&e^{\left(\frac{3}{4}(1-c^4)+c^4\log c)\right)r^4(1+o(1))}.
\end{eqnarray}
It is clear that
\begin{eqnarray}\label{eqn:est2}
&&\prod_{k=c^2r^2}^{\lambda r^2}e^{-c^2r^2+2k\log c}=\exp\left\{-c^2(\lambda-c^2)r^4+(\lambda^2-c^4)r^4\log c+O(r^2)\right\},
\\&&\prod_{k=\lambda r^2}^{r^2}e^{-r^2}=\exp{\{-r^2(r^2-\lambda r^2)\}}=\exp\{-(1-\lambda)r^4\}.\label{eqn:est3}
\end{eqnarray}
By (\ref{eqn:estimate1}), (\ref{eqn:est2}) and (\ref{eqn:est3}) from (\ref{eqn:prod1}) we get
\begin{eqnarray}\label{eqn:small}
&&\prod_{k=c^2r^2}^{\lambda r^2}\P(R_k<c^2r^2)\prod_{k=\lambda r^2}^{r^2}\P(R_k>r^2)\nonumber
\\&=&\exp\left\{\left(-\frac{1}{4}(1-c^4)+\lambda (1-c^2)+\lambda^2\log c\right)r^4(1+o(1))\right\}.
\end{eqnarray}

Now we show that $\prod_{k=1}^{c^2r^2}\P(R_k^2<c^2r^2)\prod_{k=r^2+1}^{\infty}\P(R_k^2>r^2)=e^{O(r^2)}$.
Recall that $\P(\mbox{Poisson}(a)>a)\to \frac{1}{2}$ as $a\to \infty$. Therefore for large enough $r$, we have $\P[R_k^2>r^2]\ge \frac{1}{4}$  for any $k>r^2$ and $\P(R_k^2<c^2r^2)\ge \frac{1}{4}$ for all $k\le c^2r^2$. So, for large enough $r$, we have
\begin{eqnarray}\label{eqn:middle}
&&\prod_{k=r^2+1}^{2r^2}\P(R_k^2>r^2)\ge e^{-r^2\log 4}.
\\&&\prod_{k=1}^{c^2r^2}\P(R_k^2<c^2r^2)\ge e^{-c^2r^2\log 4}.\label{eqn:first}
\end{eqnarray}
Since
\begin{eqnarray*}
\P(R_a< a/2)&=&\P(\mbox{Poisson}(a/2)> a)
\le e^{-a}\E[e^{\mbox{Poisson}(a/2)}]
=e^{-c.a},
\end{eqnarray*}
where $c=(1-e)/2$ is constant does not depend a. Therefore we have
\begin{eqnarray}\label{eqn:tail}
\prod_{k=2r^2+1}^{\infty}\P(R_k^2>r^2)&=&\prod_{k=2r^2+1}^{\infty}\P(R_k^2\ge k/2)
=\prod_{k=2r^2+1}^{\infty}(1-\P(R_k^2< k/2))\nonumber
\\&\ge&\prod_{k=2r^2+1}^{\infty}(1-e^{-c.k})=C,
\end{eqnarray}
where $C$ is a positive constant (as $\sum_{k>2r^2}e^{-c.k}<\infty$). By  (\ref{eqn:middle}), (\ref{eqn:first}) and (\ref{eqn:tail}) we get
\begin{eqnarray}\label{eqn:large}
\prod_{k=1}^{c^2r^2}\P(R_k^2<c^2r^2)\prod_{k=r^2+1}^{\infty}\P(R_k^2>r^2)=e^{O(r^2)},
\end{eqnarray}
as $r\to \infty$. By (\ref{eqn:small}) and (\ref{eqn:large}) from (\ref{eqn:initialstep}) we have
\begin{eqnarray*}\label{eqn:lowerbound}
\P[n_c(r)=0]\ge \exp\left\{\left(-\frac{1}{4}(1-c^4)+\lambda (1-c^2)+\lambda^2\log c\right)r^4(1+o(1))\right\}.
\end{eqnarray*}
as $r\to \infty$. Therefore by (\ref{eqn:upperbound}) we have
$$
\lim_{r\to \infty}\frac{1}{r^4}\log \P[n_c(r)=0]=-\frac{1}{4}(1-c^4)+\lambda (1-c^2)+\lambda^2\log c,
$$
where $\lambda=-(1-c^2)/2\log c$. Replacing the value of $\lambda $ we get
$$
\lim_{r\to \infty}\frac{1}{r^4}\log \P[n_c(r)=0]=-\frac{(1-c^2)}{4}\cdot\left(1+c^2+\frac{1-c^2}{\log c}\right).
$$
Hence the result.
\end{proof}

Now we prove Lemma \ref{lem:iqn}.
\begin{proof}[Proof of Lemma \ref{lem:iqn}]{\bf Proof of (\ref{eqn:inq1}):}
Let $X\sim \mbox{Poisson}(\lambda)$ . Then we have
\begin{eqnarray*}
\P(X>t)\le e^{-\theta t}\E(e^{\theta X})=e^{-\theta t+\lambda e^{\theta}-\lambda}.
\end{eqnarray*}
The bound is optimized for $\theta = \log (t/\lambda )$. Since $\theta>0$, hence $t>\lambda$. Therefore for $t>\lambda$, we have
$$
\P(X>t)\le e^{-t\log (t/\lambda)+t-\lambda}.
$$
In particular, for $\lambda=c^2r^2$ we get
\begin{eqnarray*}
\P(R_k^2<c^2r^2)&=&\P(\mbox{Poisson}(c^2r^2)>k)
\\&\le&e^{-k\log(k/c^2r^2)+k-c^2r^2},
\end{eqnarray*}
for $k>c^2r^2$.

\noindent{\bf Proof of (\ref{eqn:inq2}):}Since $R_k^2\sim \mbox{Gamma}(k,1)$, hence we have
$$
\P(R_k^2>r^2)\le e^{-\theta r^2}\E[e^{\theta R_k^2}]=e^{-\theta r^2}(1-\theta)^{-k}.
$$
For $k<r^2$, the bound is optimized for $\theta =1-\frac{k}{r^2}$. Therefore for $k<r^2$, the optimal bound is given by
\begin{eqnarray*}
\P(R_k^2>r^2)&\le& e^{-(1-\frac{k}{r^2})r^2-k\log (\frac{k}{r^2})}
\\&=&e^{-r^2+k-k\log (\frac{k}{r^2})}.
\end{eqnarray*}
Hence the result.
\end{proof}

\noindent{\bf Acknowledgments:} We are grateful to Manjunath Krishnapur for many
valuable suggestions and discussions. We  would like to thank Alon Nishry for pointing out mistakes in the proof of Theorem \ref{thm:generalformula} in an earlier draft. We also would like to thank Abhishek Dhar for illuminating remarks in the early stages of the project.

\noindent{\bf Funding:} This research is partially supported by UGC Centre for Advanced Studies. Research of Nanda Kishore Reddy is supported by CSIR-SPM fellowship, CSIR, Government of India.

\bibliography{Hole_IMRN-revision}
\bibliographystyle{amsplain}

\end{document}